\documentclass[11pt]{article}
\usepackage{amsmath,amssymb}
\usepackage{amsthm}
\usepackage{mathtools}
\usepackage[noend]{algorithmic}
\usepackage[ruled,vlined]{algorithm2e}
\usepackage{url}
\usepackage{fullpage}
\usepackage{makeidx}
\usepackage{enumerate}
\usepackage[top=1in, bottom=1.25in, left=1in, right=1in]{geometry}
\usepackage{graphicx,float,psfrag,epsfig,caption}
\usepackage[usenames,dvipsnames,svgnames,table]{xcolor}
\definecolor{darkgreen}{rgb}{0.0,0,0.9}
\usepackage[pagebackref,letterpaper=true,colorlinks=true,pdfpagemode=none,citecolor=OliveGreen,linkcolor=BrickRed,urlcolor=BrickRed,pdfstartview=FitH]{hyperref}
\usepackage{epstopdf}
\usepackage{color}
\usepackage{xr}
\usepackage{subfig}
\usepackage{caption}
\usepackage{graphicx}
\usepackage[utf8]{inputenc}

\usepackage{scalerel,stackengine}

\stackMath
\newcommand\reallywidehat[1]{%
\savestack{\tmpbox}{\stretchto{%
  \scaleto{%
    \scalerel*[\widthof{\ensuremath{#1}}]{\kern.1pt\mathchar"0362\kern.1pt}%
    {\rule{0ex}{\textheight}}
  }{\textheight}%
}{2.4ex}}%
\stackon[-6.9pt]{#1}{\tmpbox}%
}
\usepackage[mathscr]{euscript}

\DeclareSymbolFont{rsfs}{U}{rsfs}{m}{n}
\DeclareSymbolFontAlphabet{\mathscrsfs}{rsfs}

\numberwithin{equation}{section}

\newtheoremstyle{myexample} 
    {\topsep}                    
    {\topsep}                    
    {\rm }                   
    {}                           
    {\bf }                   
    {.}                          
    {.5em}                       
    {}  

\newtheoremstyle{myremark} 
    {\topsep}                    
    {\topsep}                    
    {\rm}                        
    {}                           
    {\bf}                        
    {.}                          
    {.5em}                       
    {}  

\newtheorem{claim}{Claim}[section]
\newtheorem{lemma}[claim]{Lemma}

\newtheorem{assumption}{Assumption}[]

\newtheorem{theorem}{Theorem}
\newtheorem{proposition}[claim]{Proposition}
\newtheorem{corollary}[claim]{Corollary}

\theoremstyle{myremark}
\newtheorem{remark}{Remark}[section]

\theoremstyle{myremark}

\theoremstyle{myexample}

\definecolor{darkgreen}{rgb}{0.0, 0.5, 0.0}

\newcommand{\bea}{\begin{eqnarray}}
\newcommand{\eea}{\end{eqnarray}}
\newcommand{\<}{\langle}
\renewcommand{\>}{\rangle}
\newcommand{\E}{{\mathbb E}}

\def\eps{{\varepsilon}}

\def\id{{\boldsymbol{I}}}

\def\supp{{\rm supp}}

\def\cuP{\mathscrsfs{P}}

\def\bsigma{{\boldsymbol{\sigma}}}

\def\Ude{U^{\delta}}
\def\Xde{X^{\delta}}
\def\Zde{Z^{\delta}}

\def\sb{{\sf b}}
\def\sB{{\sf B}}
\def\GOE{{\sf GOE}}
\def\hq{{\widehat{q}}}
\def\hQ{{\widehat{Q}}}
\def\hbQ{{\boldsymbol{\widehat{Q}}}}

\def\bg{{\boldsymbol{g}}}

\def\cG{{\mathcal G}}

\def\cC{{\mathcal C}}

\def\tz{\tilde{z}}
\def\tbz{\tilde{\boldsymbol z}}
\def\tH{\tilde{H}}

\def\op{\mbox{\tiny\rm op}}

\def\naturals{{\mathbb N}}
\def\reals{{\mathbb R}}

\def\integers{{\mathbb Z}}

\def\normal{{\sf N}}

\def\sT{{\sf T}}

\def\bv{{\boldsymbol{v}}}
\def\bz{{\boldsymbol{z}}}
\def\bx{{\boldsymbol{x}}}

\def\bA{\boldsymbol{A}}

\def\Par{{\sf P}}
\def\de{{\rm d}}

\def\bX{\boldsymbol{X}}

\def\prob{{\mathbb P}}
\def\E{{\mathbb E}}

\def\<{\langle}
\def\>{\rangle}

\def\sign{{\rm sign}}

\def\cL{{\cal L}}

\def\by{{\boldsymbol{y}}}

\def\cE{{\mathcal E}}

\def\hp{\hat{p}}
\def\toW{\stackrel{W_2}{\longrightarrow}}
\def\toP{\stackrel{p}{\longrightarrow}}
\def\toLtwo{\stackrel{L_2}{\longrightarrow}}

\def\bu{{\boldsymbol{u}}}
\def\b0{{\boldsymbol{0}}}

\def\Var{{\rm Var}}

\def\bfone{{\boldsymbol 1}}
\def\bff{{\boldsymbol f}}

\def\tbx{\tilde{\boldsymbol x}}

\DeclareMathOperator*{\plim}{p-lim}

\def\oq{\overline{q}}

\def\hg{\widehat{g}}

\def\cut{{\sf CUT}}

\def\bfzero{\boldsymbol{0}}

\title{Optimization of the Sherrington-Kirkpatrick Hamiltonian}

\author{Andrea Montanari\thanks{Department of Electrical Engineering and
  Department of Statistics, Stanford University}}

\begin{document}

\maketitle

\begin{abstract}
Let $\bA\in\reals^{n\times n}$ be a symmetric random matrix with independent and identically distributed Gaussian entries above the diagonal. 
We consider the problem of maximizing $\<\bsigma,\bA\bsigma\>$ over binary vectors $\bsigma\in\{+1,-1\}^n$. 
In the language of statistical physics, this amounts to finding the ground state of the Sherrington-Kirkpatrick model of
spin glasses.  The asymptotic value of this optimization problem was characterized by Parisi via a celebrated variational principle,
subsequently proved by Talagrand. We give an algorithm that, for any $\eps>0$, outputs $\bsigma_*\in\{-1,+1\}^n$ 
such that $\<\bsigma_*,\bA\bsigma_*\>$ is at least $(1-\eps)$ of the optimum value, with probability converging to one
as $n\to\infty$. The algorithm's time complexity is $C(\eps)\, n^2$.  It is a message-passing algorithm, but the specific structure 
of its update rules is new.

As a side result, we prove that, at (low)  non-zero temperature, the algorithm constructs approximate solutions of the 
Thouless-Anderson-Palmer equations.
\end{abstract}

\section{Introduction and main result}

Let $\bA\in\reals^{n\times n}$ be a random matrix from the $\GOE(n)$ ensemble. Namely, $\bA=\bA^{\sT}$ and
$(A_{ij})_{i\le j\le n}$ is a collection of independent random variables with $A_{ii}\sim\normal(0,2/n)$ and $A_{ij}=\normal(0,1/n)$
for $i<j$. We are concerned with the following optimization problem (here $\<\bu,\bv\> = \sum_{i\le n} u_iv_i$ is the standard scalar product)
\begin{equation}\label{eq:OPT}
\begin{split}
\mbox{maximize} &\;\; \<\bsigma,\bA\bsigma\>\, ,\\
\mbox{subject to} &\;\; \bsigma\in \{+1,-1\}^n\, .
\end{split}
\end{equation}
From a worst-case perspective, this problem is NP-hard and indeed hard to approximate 
within a sublogarithmic factor \cite{arora2005non}. For random data $\bA$,
the energy function $H_n(\bsigma) = \<\bsigma,\bA\bsigma\>/2$ is also known as the Sherrington-Kirkpatrick model \cite{sherrington1975solvable}.
Its properties have  been intensely studied in statistical physics and probability
theory for over 40 years as a prototypical example of complex energy landscape and a mean field 
model for spin glasses \cite{SpinGlass,TalagrandVolI,panchenko2013sherrington}.
Generalizations of this model have been used to understand structural glasses, random combinatorial problems, neural networks, and a number of other
systems \cite{engel2001statistical,mezard2002analytic,wolynes2012structural,NishimoriBook,MezardMontanari}.

In this paper we consider the computational problem of finding a vector $\bsigma_*\in\{+1,-1\}^n$ that is a near optimum,
namely such that $H_n(\bsigma_*) \ge (1-\eps)\max_{\bsigma\in\{+1,-1\}^n} H_n(\bsigma)$. Under a widely believed assumption about the 
structure of the associated Gibbs measure (more precisely, on the support of the asymptotic overlap distribution) we prove that, for any $\eps>0$
there exists an algorithm with complexity $O(n^2)$ that --with high probability-- outputs such a vector.

In order to state our assumption, we need to take a detour and introduce Parisi's variational formula for the value of the optimization problem
\eqref{eq:OPT}. Let $\cuP([0,1])$ be the space of probability measures on the interval $[0,1]$ endowed with the topology of weak convergence. 
For $\mu\in\cuP([0,1])$,
we will write (with a slight abuse of notation) $\mu(t) = \mu([0,t])$ for its distribution function. For $\beta\in\reals_{\ge 0}$,
consider the following parabolic partial differential equation
(PDE) on $(t,x)\in [0,1]\times \reals$
\begin{equation}\label{eq:ParisiPDE}
\begin{split}
&\partial_t\Phi(t,x) +\frac{1}{2}\beta^2\partial_{xx}\Phi(t,x) +\frac{1}{2}\beta^2\mu(t)\big(\partial_x\Phi(t,x)\big)^2 =0\, ,\\
&\Phi(1,x) = \log 2\cosh x\, .
\end{split}
\end{equation}
It is understood that this is to be solved backward in time with the given final condition at $t=1$. Existence and uniqueness where proved in 
\cite{jagannath2016dynamic}. We will also write $\Phi_\mu$ to emphasize the dependence of the solution on the measure $\mu$. 
The Parisi functional is then defined as
\begin{align}
\Par_{\beta}(\mu) \equiv \Phi_{\mu}(0,0) -\frac{1}{2}\beta^2\int_{0}^1 \!t\, \mu(t) \, \de t\, .\label{eq:ParisiFunctional}
\end{align}
The relation between this functional and the original optimization problem is given by a remarkable variational
principle, first proposed by Parisi \cite{parisi1979infinite} and
established rigorously, more than twenty-five years later, by Talagrand \cite{talagrand2006parisi}, and 
Panchenko \cite{panchenko2013parisi}.
\begin{theorem}[Talagrand \cite{talagrand2006parisi}]\label{thm:Talagrand}
Consider the partition function $Z_n(\beta) = \sum_{\bsigma\in\{+1,-1\}^n}\exp\{\beta H_n(\bsigma)\}$. 
Then we have, almost surely (and in $L^1$)
\begin{align}
\lim_{n\to\infty} \frac{1}{n}\log Z_n(\beta) = \min_{\mu\in\cuP([0,1])}\Par_{\beta}(\mu)\, .\label{eq:Parisi}
\end{align}
\end{theorem}

The following consequence for the optimization problem \eqref{eq:OPT} is elementary, see e.g. \cite{dembo2017extremal}.
\begin{corollary}\label{coro:T0}
We have, almost surely
\begin{align}
\lim_{n\to\infty} \frac{1}{2n}\max_{\bsigma\in\{+1,-1\}^n}\<\bsigma,\bA\bsigma\> = \lim_{\beta\to\infty}\frac{1}{\beta}\min_{\mu\in\cuP([0,1])}\Par_{\beta}(\mu)\, .
\label{eq:ParisiT0}
\end{align}
\end{corollary}

\begin{remark}
The limit $\beta\to\infty$ on the right-hand side of Eq.~\eqref{eq:ParisiT0} can be removed by defining a new variational principle directly 
`at $\beta=\infty$'. Namely, the right-hand side of Eq.~\eqref{eq:ParisiT0} can be replaced by $\min_{\gamma}\hat{\Par}(\gamma)$ where
$\hat{\Par}$ is a modification of $\Par$ and the minimum is taked over a suitable functional space \cite{auffinger2017parisi}. In this paper
 we use the $\beta<\infty$ formulation, but it should be possible
to work directly at $\beta=\infty$.
\end{remark}

Existence and uniqueness of the minimizer of $\Par_{\beta}(\,\cdot\,)$ were proved in \cite{auffinger2015parisi} and \cite{jagannath2016dynamic},
which also proved that $\mu\mapsto \Par_{\beta}(\mu)$ is strongly convex. We will denote by $\mu_{\beta}$ the unique minimizer, and refer to it as
the `Parisi measure' or `overlap distribution' at inverse temperature $\beta$. 
Our key assumption will be that --at large enough $\beta$-- the support of $\mu_{\beta}$ is an interval $[0,q_*(\beta)]$. 
\begin{assumption}\label{ass:FRSB}
There exist $\beta_0<\infty$ such that, for any $\beta> \beta_0$, the function $t\mapsto \mu_{\beta}([0,t])$ is strictly increasing on $[0,q_*]$,
where $q_*=q_*(\beta)$ and $\mu_{\beta}([0,q_*]) = 1$.
\end{assumption}
This assumption (sometimes referred to as `continuous replica symmetry breaking' or `full replica symmetry breaking') is widely believed to
be true (with $\beta_0=1$) within statistical physics \cite{SpinGlass}. In particular, this conjecture 
is supported by high precision numerical solutions of the variational problem
for $\Par_{\beta}$\cite{crisanti2002analysis,oppermann2007double,schmidt2008method}. Rigorous evidence was recently obtained
in \cite{auffinger2017sk}. Addressing this conjecture goes beyond the scope of the present paper.

We are now in position to state our main result.
\begin{theorem}\label{thm:Main}
Under Assumption \ref{ass:FRSB}, for any $\eps>0$ there exists an algorithm that takes as input the matrix $\bA\in\reals^{n\times n}$, and
outputs $\bsigma_*=\bsigma_*(\bA)\in\{+1,-1\}^n$, such that the following hold:
\begin{itemize}
\item[$(i)$] The complexity (floating point operations) of the algorithm is at most $C(\eps)n^2$.
\item[$(ii)$] We have $\<\bsigma_*,\bA\bsigma_*\>\ge (1-\eps) \max_{\bsigma\in\{+1,-1\}^n} \<\bsigma,\bA\bsigma\>$, 
with high probability (with respect to $\bA\sim\GOE(n)$).
\end{itemize}
\end{theorem}
In other words, on average, the optimization problem \eqref{eq:OPT} is much easier than in worst case. Of course, this is far from being the only example
of this phenomenon (a gap between worst case and average case complexity). However,
 it is a rather surprising example given the complexity of the energy landscape $H_n(\bsigma)$. Its proof uses in a crucial
way a fine property of the associated Gibbs measure, namely the support overlap distribution.
\begin{remark}[Computation model]
For the sake of simplicity, we measure complexity in floating point operations. However, all operations in our algorithm appear to be
stable and it should be possible to translate this result to weaker computation models.

 We also assume that we can choose one value of the inverse temperature $\beta$, and query the distribution $\mu_{\beta}(t)$ and 
the PDE solution $\Phi(t,x)$ as well as its derivatives
$\partial_x\Phi(t,x)$, $\partial_{xx}\Phi(t,x)$ at specified points $(t,x)$, with each query costing $O(1)$ operations. 

This is a reasonable model for two reasons: $(i)$~The PDE
\eqref{eq:ParisiPDE} is independent of the instance, and can be solved
to a desired degree of accuracy only
once.  This solution can be used every time a new instance of the problem is presented. 
$(ii)$~The function $\mu\mapsto \Par_{\beta}(\mu)$ is uniformly continuous
\cite{guerra2003broken} and strongly convex \cite{auffinger2015parisi,jagannath2016dynamic}. Further the PDE solution $\Phi$ is continuous in $\mu$
 and can be characterized as fixed point of a certain contraction \cite{jagannath2016dynamic}. Because of these reasons we expect that an oracle  to compute
$\Phi(t,x)$, $\partial_x\Phi(t,x)$, $\partial_{xx}\Phi(t,x)$ to accuracy $\eta$ can be implemented in $O(\eta^{-C})$ operations, for $C$ a constant.
\end{remark}

Beyond Theorem \ref{thm:Main}, our general analysis allows to prove an additional fact that is of independent interest.
Namely, for any $\beta>\beta_0$, our message passing iteration constructs an approximate solution of the 
celebrated Thouless, Anderson, Palmer (TAP) equations \cite{SpinGlass,TalagrandVolI}.

The bulk of this paper is devoted to the case of Gaussian matrices $\bA$. However, the class  of algorithms we use
enjoys certain universality properties, first established in \cite{bayati2015universality}. These properties can be used to generalize 
Theorem \ref{thm:Main} to the case of symmetric matrices with independent subgaussian entries. We refer to Section
\ref{sec:Universality} for a statement of of this universality result, and limit ourselves to state a consequence of Theorem \ref{thm:Main}
for the MAXCUT problem.

Let $G_n= ([n],E_n)\sim\cG(n,p)$ be an Erd\"os-Renyi random graph with edge
probability $\prob\big\{(i,j)\in E_n\big\}=p$. A random balanced partition of the vertices (which we encode as a vector $\bsigma\in\{+1,-1\}^n$)
achieves a cut $\cut_G(\bsigma) = |E_n|/2 + O(n) =n^2p/4+O(n)$, and simple concentration argument implies that the MAXCUT has size
$\max_{\bsigma\in\{+1,-1\}}\cut_G(\bsigma) = |E_n|/2 + O(n^{3/2}p^{1/2})$.
In fact, it follows from \cite{dembo2017extremal} that\footnote{In \cite{dembo2017extremal}, the same result is shown to hold for sparser graphs,
as long as the average degree diverges: $np_n\to\infty$.}
 $\max_{\bsigma\in\{+1,-1\}^n}\cut_G(\bsigma)= |E_n|/2+(n^{3}p(1-p)/2)^{1/2}\Par_*+o(n^{3/2})$,
where $\Par_*$ is the prediction of Parisi's formula (i.e. the right-hand side of (\eqref{eq:Parisi})).
 In other words, MAXCUT on dense Erd\"os-Renyi random graphs is non-trivial only once we 
subtract the baseline value $|E_n|/2$. As a corollary of Theorem \ref{thm:Main} we can approximate this subtracted value arbitrarily well.
\begin{corollary}\label{coro:MaxCut}
Under Assumption \ref{ass:FRSB}, for any $\eps>0$ there exists an algorithm (with complexity at most $C(\eps)\, n^2$),
 that takes as input an Erd\"os-Renyi random graph $G_n=([n],E_n)\sim\cG(n,p)$, and outputs $\bsigma_*=\bsigma_*(G)\in\{+1,-1\}^n$, such that
\begin{align}
\left(\cut_G(\bsigma_*)-\frac{|E_n|}{2}\right)\ge (1-\eps)\max_{\bsigma\in\{+1,-1\}^n}\left(\cut_G(\bsigma_*)-\frac{|E_n|}{2}\right)\, .
\end{align} 
\end{corollary}

The rest of this section provides further background. 
In Section \ref{sec:Algo} we describe and analyze a general message passing algorithm, which 
we call \emph{incremental approximate message passing} (IAMP). We believe this algorithm  is of independent interest
and can be applied beyond the Sherrington-Kirkpatrick model.  In Section \ref{sec:Proof} we use this approach to prove Theorem \ref{thm:Main}.
In Section \ref{sec:TAP} we show that the same message passing algorithm of Section \ref{sec:Algo} produces approximate solutions of the TAP 
equations. Finally, Section \ref{sec:Universality} discusses a generalization of Theorem \ref{thm:Main} using universality.
The impatient reader, who is interested in a succinct description of the algorithm (with some technical bells and whistles removed),
is urged to read Appendix \ref{app:Simplified}.

\subsection{Further background}

As mentioned above  --under suitable complexity theory assumptions-- there is mo polynomial-time
algorithm that  approximates the quadratic program \eqref{eq:OPT} better than within a factor $O((\log n)^c)$, for some $c>0$  \cite{arora2005non}.
Little is known on average-case hardness, when $\bA$ is drawn from one of the random matrix distributions considered here. 
As an exception, Gamarnik \cite{gamarnik2018computing} proved that exact computation of the partition function $Z_n(\beta)$ is hard on average.

A natural approach to the quadratic program \eqref{eq:OPT} would be to use a convex relaxation. A spectral relaxation yields
$\max_{\bsigma\in\{+1,-1\}}H_n(\bsigma)/n\le \lambda_1(\bA)/2 = 1+o_n(1)$, and hence is not tight for large $n$. This can be compared to a numerical
evaluation of Parisi's formula which yields $\Par_*\approx 0.763166$ \cite{crisanti2002analysis,schmidt2008replica}. 
Rounding the spectral solution yields a $H_n(\bsigma_{\mbox{\tiny sp}}) = 2/\pi +o_n(1)\approx 0.636619$. Somewhat surprisingly, the simplest 
semidefinite programming relaxation (degree $2$ of the sum-of-squares hyerarchy), does not yield any improvement (for large $n$) over the spectral one
\cite{montanari2016semidefinite}. After an initial version of this paper was posted, \cite{bandeira2019computational} obtained rigorous
evidence that higher order relaxations fail as well.

Theorem \ref{thm:Main} was conjectured by the author in 2016 \cite{MyTalk2016}, based on insights from statistical physics
\cite{cugliandolo1994out,bouchaud1998out}. The same presentation also outlined the basic strategy followed in the present paper, which uses
an iterative `approximate message passing'  (AMP)  algorithm. 
This type of algorithms were first proposed in the context of signal processing and compressed sensing \cite{Kab03,DMM09}.
Their rigorous analysis was developed by Bolthausen \cite{bolthausen2014iterative}
and subsequently generalized in several papers \cite{BM-MPCS-2011,javanmard2013state,bayati2015universality,berthier2017state}. 
In this paper we introduce a specific class of AMP algorithms (`incremental AMP') whose specific properties allow us to match the result predicted by Parisi's formula.

The fundamental phenomenon studied here is expected to be quite general.  Namely objective functions 
with overlap distribution having support of the form  $[0,q_*]$ are expected to  be easy to optimize.
In contrast, if the support has a gap (for instance, has the form $[0,q_1]\cup [q_2,q_*]$ for some $q_1<q_2$), this is considered as an indication
of average case  hardness.
This intuition originates within spin glass theory \cite{SpinGlass}. Roughly speaking, the structure of the overlap distribution 
should reflect the connectivity properties of  the level sets 
$\cL_n(\eps) \equiv \{\bsigma:\; H_n(\bsigma)\ge (1-\eps)\max_{\bsigma'}H_n(\bsigma')\}$. This intuition
was exploited in some cases to prove the failure of certain classes of algorithms in problems
with a gap in the overlap distribution, see e.g. \cite{gamarnik2014limits}.

 Important  progress towards clarifying this connection was achieved recently in two remarkable papers
\cite{addario2018algorithmic,subag2018following}. 

Addario-Berry and Maillard \cite{addario2018algorithmic} study an abstract
optimization problem that is thought to capture some key features of the the energy landscape of the Sherrington-Kirkpatrick model, the
so-called `continuous random energy model.'  They prove that an approximate optimum can be found in time polynomial in the problem dimensions.
From an optimization perspective, the random energy model is somewhat un-natural, in that specifying an instance requires memory that
is exponential in the problem dimensions.

Subag \cite{subag2018following} considers the $p$-spin spherical spin glass. Roughly speaking, this can be described as the problem of optimizing 
a random smooth function (which can be taken to be a low-degree polynomial) over the unit sphere. Subag relaxes this problem by extending the optimization
over the unit ball, and proves that this objective function can be optimized efficiently by following the positive directions of the Hessian. 
The solution thus constructed lies on the unit sphere and thus solves the un-relaxed problem.
The mathematical insight 
of \cite{subag2018following} is beautifully simple, but uses in a crucial way the spherical geometry. While it might be possible to generalize
the same argument to the hypercube case (e.g., using the generalized TAP free energy of \cite{mezard1985microstructure,chen2018generalized}) 
this extension is far from obvious. In particular, uniform control of the Hessian is not as straightforward as in  \cite{subag2018following}.

The algorithm presented here is partially inspired by  \cite{subag2018following} (in particular, a key role is played by 
approximate orthogonality of the updates), but its specific structure is dictated
by the message passing viewpoint. Thanks to the technique of \cite{bolthausen2014iterative,BM-MPCS-2011,javanmard2013state,berthier2017state}, its analysis 
does not require uniform control and is relatively simple. 
 
\subsection{Notations}

Given  vectors $\bx,\by\in\reals^n$, we denote by $\<\bx,\by\>$ their scalar product and by $|\bx|\equiv \<\bx,\bx\>^{1/2}$
the associated $\ell_2$ norm. 
Given a function $f:\reals^k\to\reals$, and $k$ vectors $\bx_1,\dots,\bx_k\in\reals^n$ we write $f(\bx_1,\dots,\bx_k)$ for
the vector in $\reals^n$ with components $f(\bx_1,\dots,\bx_k)_i =
f(x_{1,i},\dots,x_{k,i})$.
The empirical distribution of the coordinates of a  vector of vectors $(\bx_1,\dots,\bx_k)\in(\reals^n)^k$ is the
probability measure on $\reals^k$ defined by
\begin{align}
\hp_{\bx_1,\dots,\bx_k} \equiv \frac{1}{n}\sum_{i=1}^n\delta_{(x_{1,i},\dots,x_{k,i})}\, .
\end{align}
In other words, if we arrange the vectors $\bx_1,\dots,\bx_k$ in a matrix in $\bX = [\bx_1,\dots,\bx_k]\in \reals^{n\times k}$, $\hp_{\bx_1,\dots,\bx_k}$ denotes the probability
distribution of a uniformly random row of $\bX$. 
In the case of a single vector $\bx\in\reals^n$ (i.e. for $k=1$), this reduces to the standard empirical distribution of the entries of $\bx$.
We say that a function $f:\reals^d\to\reals$ is pseudo-Lipschitz if
$|f(\bx)-f(\by)|\le C(1+|\bx|+|\by|)|\bx-\by|$.

Given two probability measures $\mu$, $\nu$ on $\reals^d$, we recall that their Wasserstein $W_2$ distance is defined as
\begin{align}
W_2(\mu,\nu) \equiv \left\{\inf_{\gamma\in \cC(\mu,\nu)} \int |\bx-\by|^2\gamma(\de\bx,\de\by)\right\}^{1/2}\, ,
\end{align}
where the infimum is taken over all the couplings of $\mu$ and $\nu$ (i.e. joint distributions on $\reals^d\times\reals^d$ whose first marginal 
coincides with $\mu$, and second with $\nu$. 
For a sequence of probability measures $(\mu_n)_{n\ge 1}$, and $\mu$ on $\reals^d$, we say that $\mu_n$ converges in Wasserstein distance to 
$\mu$ (and write $\mu_n\toW \mu$) if $\lim_{n\to\infty}W_2(\mu_n,\mu) = 0$. It is well known   that $\mu_n\toW \mu$ if and
only if $\lim_{n\to\infty}\int \psi(\bx)\mu_n(\de\bx) = \int \psi(\bx)\mu(\de\bx)$ for all bounded Lipschitz functions $\psi$, and for $\psi(\bx) = |\bx|^2$
\cite[Theorem 6.9]{villani2008optimal}.
Given a sequence of random variables $X_n$, we write $X_n\toP X_{\infty}$ or $\plim_{n\to\infty}X_n =  X_{\infty}$ 
to state that $X_n$ converge in probability to $X_{\infty}$. 

We will sometimes be interested in double limits of sequences of random  variables. 
If $X_{n,M}$ is a sequence indexed by $n,M$ and $x_{*}$ is a constant,
\begin{align}
\lim_{M\to\infty}\plim_{n\to \infty}X_{n,M} = x_*\, ,
\end{align}
whenever $X_{n,M}$ converges in probability to a non-random quantity $x_M$ as $n\to\infty$, and
$\lim_{M\to\infty} x_M=x_*$.

\section{A general message passing algorithm}
\label{sec:Algo}

Our algorithm is based on the following approximate message passing (AMP) iteration. 
\begin{description}
\item[AMP iteration] Consider a sequence of (weakly differentiable) functions $f_k:\reals^{k+2}\to \reals$, and a non-random initialization $\bu^0\in\reals^n$
and additional vector $\by\in\reals^d$ with $\hp_{\bu^0,\by}\toW p_{U_0,Y}$ (where $p_{U_0,Y}$ is any probability distribution on $\reals^2$ with
finite second moment $\int (u_0^2+y^2)p_{U_0,Y}(\de u_0,\de y)<\infty$). The AMP iteration is defined by letting, for $k\ge 0$, 
\begin{equation}\label{eq:GeneralAMP}
\begin{split}
\bu^{k+1} &= \bA \, f_k(\bu^0,\dots,\bu^k;\by) - \sum_{j=1}^k\sb_{k,j} f_{j-1}(\bu^0,\dots,\bu^{j-1};\by)\, ,\\
\sb_{k,j}& = \frac{1}{n}\sum_{i=1}^n\frac{\partial f_k}{\partial u^j_i}(u_i^0,\dots,u_i^k;y_i) \, .
\end{split}
\end{equation}
It will be understood throughout that $f_{j} = 0$ for $j<0$.
\end{description}

\begin{proposition}\label{propo:GeneralAMP}
Consider the AMP iteration \eqref{eq:GeneralAMP}, and assume $f_k:\reals^{k+2}\to \reals$ to be Lipschitz continuous.
Then for any $k\in\naturals$, and any pseudo-Lipschitz function $\psi:\reals^{k+2}\to \reals$, we have
\begin{align}
\frac{1}{n}\sum_{i=1}^n\psi(u_i^0,\dots,u_i^k;y_i) \toP \E \psi(U_0,\dots,U_k;Y)\, .\label{eq:LimitGeneralAMP}
\end{align}
Here $(U_j)_{j\ge 1}$ is a centered Gaussian process independent of $(U_0,Y)$ with covariance $\hbQ = (\hQ_{kj})_{k,j\ge 1}$ determined recursively 
via
\begin{align}
\hQ_{k+1,j+1} = \E\big\{f_k(U_0,\dots,U_k;Y) f_j(U_0,\dots,U_j;Y)\big\}\, . \label{eq:SE-AMP}
\end{align}
\end{proposition}
This proposition follows immediately from the general analysis of AMP algorithms developed in \cite{javanmard2013state,berthier2017state},
cf. Appendix \ref{app:GeneralAMP}.

We next consider a special case of the general AMP setting.
\begin{description}
\item[Incremental AMP (IAMP)] Fix $\delta, M>0$, and  functions 
$\hg_k:\reals\to \reals$, $k\in\naturals$, $s, v:\reals\times \reals_{\ge 0}\to \reals$. We consider the general iteration
\eqref{eq:GeneralAMP}, with the following choice of functions $f_k$ (independent of $y$):
\begin{align}
f_k(u_0,\dots,u_k) &= \hg_k\left(x_{k-1} \right)\, \cdot [u_k]_M\, ,\label{eq:IAMP-Def-1}\\
x_k& = x_{k-1} + v(x_{k-1},k\delta)\, \delta + s(x_{k-1},k\delta)\, [u_k]_M\sqrt{\delta}\, ,\;\;\;\; x_0 = 1\, ,\label{eq:IAMP-Def-2}
\end{align}
where, for $u\in\reals$, $[u]_M = \max(-M,\min(u,M))$. Following our convention for $f_j$, we set $\hg_j=0$ for $j<0$.
\end{description}
We note that, by Eq.~\eqref{eq:IAMP-Def-1}, $x_k$ is indeed a function of $u_0,\dots,u_k$, and therefore $f_k$
is a function of $u_0,\dots,u_k$ as stated.

\begin{lemma}[State evolution for Incremental AMP]\label{lemma:SE}
Consider the incremental AMP iteration, and assume $g, s, v:\reals\times \reals_{\ge 0}\to \reals$ to be Lipschitz continuous and bounded. 
Then for any $k\in\naturals$, and any pseudo-Lipschitz function $\psi:\reals^{k+2}\to \reals$, we have
\begin{align}
\lim_{M\to\infty}\plim_{n\to\infty}\frac{1}{n}\sum_{i=1}^n\psi(u_i^0,\dots,u_i^k)  =\E \psi(\Ude_0,\dots,\Ude_k)\, .
\end{align}
(The double limit is to be interpreted as defined in the Notations section.)
Here $(\Ude_j)_j\ge 1$ is a centered Gaussian process independent of $\Ude_0=U_0$, with independent entries,
with variance $\Var(\Ude_{k}) = \hq_k$ given recursively by
\begin{equation}\label{eq:SE-IAMP}
\begin{split}
\hq_{k+1} &= \E\{\hg_k(\Xde_{k-1})^{2}\} \cdot\hq_k\, , \\
 \Xde_k &= \Xde_{k-1}+v(\Xde_{k-1};k\delta)\, \delta+s(X_{k-1};k\delta)\Ude_k\sqrt{\delta}\, . 
\end{split}
\end{equation}
\end{lemma}
\begin{proof}
Consider Eqs.~\eqref{eq:IAMP-Def-1}, \eqref{eq:IAMP-Def-2}, and note that, for any $k$, $x_{k-1}$ is a bounded Lipschitz function
of $u_0,\dots,u_{k-1}$ (because bounded Lipschitz functions are closed under sum, product, and composition).  Hence $f_k$ defined 
in \eqref{eq:IAMP-Def-1} is Lipschitz continuous and we can therefore apply Proposition \ref{propo:GeneralAMP} to get
\begin{align}
\frac{1}{n}\sum_{i=1}^n\psi(u_i^0,\dots,u_i^k)  \toP \E \psi(U^{\delta,M}_0,\dots,U^{\delta,M}_k)\, .
\end{align}
Here $(U^{\delta,M}_{j})_{j\ge 1}$ is a Gaussian process with covariance $\hbQ^M$ determined by Eq.~\eqref{eq:SE-AMP}. 
We next claim the following:
\begin{enumerate}
\item $\hQ^M_{j,k}=0$ for $k\neq j$ (and we set $\hq^M_k\equiv\hQ^M_{k,k}$).
\item $\hq^M_k\to \hq_k$ for each $k$ as $M\to\infty$.
\end{enumerate}
With these two claims, the statement of the lemma follows by dominated convergence.

To prove claim 1 note that, by symmetry we only have to consider the case $j<k$. The proof is by induction over $k$. 
For $k=1$ there is nothing to prove. Assume next that the claim holds up to a certain $k$, and consider
$\hQ^M_{j,k+1}$ for $1\le j\le k$.  By Eq.~\eqref{eq:SE-AMP} we have (dropping for simplicity the superscripts $\delta,M$
from the random variables)
\begin{align}
\hQ^M_{j,k+1} & = \E\big\{\hg_{j-1}(X_{j-2})[U_{j-1}]_M \, \hg_k(X_{k-1})[U_k]_M\big\}\\
& = \E\big\{\hg_{j-1}(X_{j-2})[U_{j-1}]_M \, \hg_k(X_{k-1})\big\}\E\big\{[U_k]_M\big\} = 0\, .
\end{align}
Here the second equality follows from the induction hypothesis.

To prove claim 2, note that $\hq_k^M$ satisfies the recursion that follows from Eq.~\eqref{eq:SE-AMP}, namely
\begin{align}
\hq^M_{k+1} &= \E\{\hg_k(X^{\delta,M}_{k-1})^{2}\} \cdot\E\{[U^{\delta,M}_k]_M^2\}\, , \\
 X^{\delta,M}_k &= X^{\delta,M}_{k-1}+v(X^{\delta,M}_{k-1};k\delta)\, \delta+s(X^{\delta,M}_{k-1};k\delta)\, [U^{\delta,M}_k]_M\sqrt{\delta}\, . 
\end{align}
Also note that $|\hg_k(X^{\delta,M}_{k-1})|\le F_k(U_0,U_1^{\delta,M},\dots,U_{k-1}^{\delta,M})$ for some polynomial $F_k$ independent of $M$.
Hence the claim follows by applying recursively dominated convergence.
\end{proof}

\begin{remark}
The use of truncation $[u_k]_M$ in the definition \eqref{eq:IAMP-Def-1} is dictated by the need to ensure that $f_k$
is Lipschitz, and to be able to apply Proposition \ref{propo:GeneralAMP}. We believe that the conclusion of  Proposition \ref{propo:GeneralAMP}
holds under weaker assumptions (e.g. $f_k$ locally Lipschitz with polynomial growth). 
Such a generalization would allow to  replace  $[u_k]_M$ by $u_k$ in Eq.~\eqref{eq:IAMP-Def-1}, and hence get rid of the parameter $M$
in our algorithm.
\end{remark}

We are now in position of defining our candidate for a near optimum of the problem 
\eqref{eq:OPT}. We fix $\oq>0$ and define (recalling the definition of $f_k$ in Eqs.~\eqref{eq:IAMP-Def-1}, \eqref{eq:IAMP-Def-2}) 
\begin{align}\label{eq:z-def}
\bz = \sqrt{\delta}\sum_{k=1}^{\lfloor \oq/\delta\rfloor} f_k(\bu_0,\dots,\bu_k)\,.
\end{align}
Note that this vector depends on parameters $\delta, M, \oq$, and on the functions $g, s, v$. Parameters $\delta$ and $M$ must be taken
(respectively) small enough and large enough (but independent of $n$). The next section will be devoted to choosing $\oq$ and the functions $g, s, v$. 
In this section we will establish some general properties of $\bz$ (for small $\delta$ and large $M$).

\begin{lemma}\label{lemma:Zvalue}
Consider the incremental AMP iteration, and assume $g, s, v:\reals\times \reals_{\ge 0}\to \reals$ to be Lipschitz continuous and bounded.
Further assume $\partial_x \hg_k(x)$,  $\partial_x s(x,t)$, $\partial_x v(x,t)$ to exist and be Lipschitz continuous. Define
the random variable 
\begin{align}
Z^{\delta}\equiv \sqrt{\delta}\sum_{k=1}^{\lfloor \oq/\delta\rfloor} \hg_k(X_{k-1})\, U^{\delta}_k\, .
\end{align}
Then we have, for any pseudo-Lipschitz function $\psi:\reals\to\reals$,
\begin{align}
&\lim_{M\to\infty}\plim_{n\to\infty} \frac{1}{n}\sum_{i=1}^n\psi(z_i) = \E\{\psi(Z^{\delta})\}\, ,\label{eq:LimPsiZ}\\
&\lim_{M\to\infty}\plim_{n\to\infty}\frac{1}{2n}\<\bz,\bA\bz\> = \delta\sum_{k=1}^{\lfloor \oq/\delta\rfloor-1}
\E\{ (U^{\delta}_k)^2\} \, \E\{\hg_k(X^{\delta}_{k-1})\}\, \E\{\hg_k(X^{\delta}_{k-1})^2\}\, . \label{eq:LimZAZ}
\end{align}
\end{lemma}
\begin{proof}
Equation \eqref{eq:LimPsiZ} follows immediately from Lemma \ref{lemma:SE} upon noticing that $\psi(z_i)$ is a pseudo-Lipschitz
function of $u_{0,i}$, \dots, $u_{k,i}$. 

In order to prove Eq.~\eqref{eq:LimZAZ}, we will write $\bff_k = f_k(\bu_0,\dots,\bu_k)$, and $K =\lfloor \oq/\delta\rfloor$. 
We further notice that, for $j<k$,
\begin{align}
\plim_{n\to\infty}\sb_{k,j}&= \plim_{n\to\infty}\frac{1}{n} \sum_{i=1}^n\frac{\partial g}{\partial u_j}(x^{k-1}_i;k\delta) [u^k_i]_M\\
& = \E\left\{\frac{\partial g}{\partial u_j}(X^{\delta,M}_{k-1};k\delta) [U_k^{\delta,M}]_M\right\}\\
& =\E\left\{\frac{\partial g}{\partial u_j}(X^{\delta,M}_{k-1};k\delta) \right\}\E\left\{[U_k^{\delta,M}]_M\right\} = 0\, . \label{eq:Bkj}
\end{align} 
Here and below the random variables $U_k^{\delta,M}, X_k^{\delta,M}$ are defined as in the proof of Lemma \ref{lemma:SE}.
On the other hand
\begin{align}
\plim_{n\to\infty}\sb_{k,k} &=\plim_{n\to\infty}\frac{1}{n} \sum_{i=1}^n \hg_k(x^{k-1}_i) \bfone_{u^k_i\in[-M,M]}\\
& = \E\{\hg_k(X^{\delta,M}_{k-1}) \}\, \prob(U_k^{\delta,M}\in [-M,M])\, .
\end{align}
Note that we applied Lemma \ref{lemma:SE} to a non-Lipschitz function. The limit holds nevertheless by a standard weak convergence
argument  (namely, using upper and lower Lipschitz approximations of the indicator function). We therefore conclude that
(using $U_k^{\delta,M}\sim\hq_k^M$, $\hq_k^M\to \hq_k$ as $M\to\infty$):
\begin{align}
\lim_{M\to\infty}\plim_{n\to\infty}\sb_{k,k} &=\E\{\hg_k(\Xde_{k-1}) \}\, . \label{eq:Bkk}
\end{align}
Next notice that, for $j<k$, 
\begin{align}
\plim_{n\to\infty} \frac{1}{n}\<\bff_j,\bff_k\> &= \E\big\{\hg_j(X^{M,\delta}_{j-1})\, [U_j^{M,\delta}]_M \hg_k(X^{M,\delta}_{k-1})\, [U_k^{M,\delta}]_M\big\}\\
& = \E\big\{\hg_j(X^{M,\delta}_{j-1})\, [U_j^{M,\delta}]_M \hg_k(X^{M,\delta}_{k-1})\big\}\E\big\{[U_k^{M,\delta}]_M\big\} = 0\, .
\end{align}
By a similar argument, always for $j<k$,
\begin{align}
\plim_{n\to\infty} \frac{1}{n}\<\bff_j,\bu^k\> &= 0\, .
\end{align}
On the other hand 
\begin{align}
\lim_{M\to\infty}\plim_{n\to\infty} \frac{1}{n}|\bff_k|^2 &= \lim_{M\to\infty}\E\big\{\hg_k(X^{M,\delta}_{k-1})^2\, [U_k^{M,\delta}]^2_M\big\}\\
&=\lim_{M\to\infty} \E\big\{\hg_k(X^{M,\delta}_{k-1})^2\big\}\, \E\big\{[U_k^{M,\delta}]^2_M\big\}\\
& = \E\big\{\hg_k(\Xde_{k-1})^2\big\}\, \E\big\{(\Ude_k)^2\big\}\,.
\end{align}

By the AMP iteration, we know that $\bA\bff_k = \bu^{k+1}+\sum_{\ell=1}^k\sb_{k,\ell}\bff_{\ell-1}$. 
Hence, using the above limits, for $j\le k$,
\begin{align}
\lim_{M\to\infty}\plim_{n\to\infty}\frac{1}{n}\<\bff_j,\bA\bff_k\> &= \lim_{M\to\infty}\plim_{n\to\infty}\frac{1}{n}\<\bff_j,\bu^{k+1}\>+
\sum_{\ell=1}^k  \lim_{M\to\infty}\plim_{n\to\infty} \sb_{k,\ell}\<\bff_j,\bff_{\ell-1}\>\\
& = \lim_{M\to\infty}\plim_{n\to\infty} \sb_{k,k}\<\bff_j,\bff_{k-1}\>\\
& = \bfone_{\{k=j+1\}}\E\big\{\hg_j(\Xde_{j-1})^2\big\}\, \E\big\{(\Ude_j)^2\big\}\E\{\hg_j(\Xde_{j-1}) \}\, .
\end{align}
We finally can compute
\begin{align}
\lim_{M\to\infty}\plim_{n\to\infty}\frac{1}{2n}\<\bz,\bA\bz\> &=
\delta\sum_{j=1}^K\lim_{M\to\infty}\plim_{n\to\infty}\frac{1}{2n}\<\bff_j,\bA\bff_j\> +
\delta\sum_{1\le j<k\le K}\lim_{M\to\infty}\plim_{n\to\infty}\frac{1}{n}\<\bff_j,\bA\bff_k\>\\
& =  \delta\sum_{j=1}^{K-1}\lim_{M\to\infty}\plim_{n\to\infty}\frac{1}{n}\<\bff_j,\bA\bff_{j+1}\>\\
& = \delta\sum_{j=1}^{K-1} \E\big\{\hg_j(\Xde_{j-1})^2\big\}\, \E\big\{(\Ude_j)^2\big\}\E\{\hg_j(\Xde_{j-1}) \}\, .
\end{align}
\end{proof}

In the case of models with full replica symmetry breaking, it is natural to consider the limit of small step size $\delta\to 0$.
This limit is described by a stochastic differential equation (SDE) described below.
\begin{description}
\item[SDE description.] Consider  Lipschitz functions $g,s,v:\reals\times\reals_{\ge 0}\to \reals$, with $|s(x,t)|+|v(x,t)|\le C(1+|x|)$. 
Let $(B_t)_{t\ge 0}$ be a standard Brownian motion.
We define the process $(X_t,Z_t)_{t\ge 0}$ via
\begin{equation}\label{eq:SDE}
\de X_t = v(X_t,t)\, \de t + s(X_t,t)\, \de B_t\, ,\;\;\;\;\;  \de Z_t = g(X_t,t)\, \de B_t\,  ,
\end{equation}
with initial condition $X_0=Z_0=0$. Equivalently 
\begin{equation}\label{eq:SDE-2}
X_t = \int_0^tv(X_r,r)\, \de r + \int_0^ts(X_r,r)\, \de B_r\, ,\;\;\;\;\;  Z_t = \int_0^tg(X_r,r)\, \de r\,  ,
\end{equation}
where the integral is understood in Ito's sense. Existence and uniqueness of strong solutions of this SDE is given --for instance-- 
in \cite[Theorem 5.2.1]{oksendal2003stochastic}.
\end{description}

\begin{lemma}\label{lemma:SDE}
Given  Lipschitz functions $g,s,v:\reals\times\reals_{\ge 0}\to \reals$, with $v$ and $s$ bounded, let $(X_t,Z_t)$ be the process defined above.
Assume $\E\{g(X_t,t)^2\} = 1$ for all $t\ge 0$. Further consider the state evolution iteration of Eq.~\eqref{eq:SE-IAMP}, whereby $\hg_k$ is defined recursively 
via
\begin{align}
\hg_k(x) \equiv \frac{g(x,k\delta)}{\E\{g(\Xde_{k-1},k\delta)^2\}^{1/2}}\, .\label{eq:hg-def}
\end{align}
Then, there exists a coupling of $(\Xde_k)_{k\ge 0}$ and $(X_t)_{t\ge 0}$ such that
\begin{align}
&\max_{k\le \lfloor \oq/\delta\rfloor} \E\big(|\Xde_{k}-X_{k\delta}|^2\big)\Big\}\le C\delta\, ,\label{eq:Discr-1-a}\\
&\E\big(|\Zde-Z_{\oq}|^2\big)\le C\, \sqrt{\delta}\, ,\label{eq:Discr-1-b}\\
&\delta\sum_{k=1}^{\lfloor \oq/\delta\rfloor-1}
\E\{ (U^{\delta}_k)^2\} \, \E\{\hg_k(X^{\delta}_{k-1})\}\, \E\{\hg_k(X^{\delta}_{k-1})^2\} = 
\int_0^{\oq} \E\{g(X_t,t)\} \, \de t +O(\delta^{1/4})\, .\label{eq:Discr-2}
\end{align}
(Here $C$ is a constant depending only on the bounds on $g,v,s$, and on $\oq$. Further the $O(\delta^{1/4})$ error is bounded as
$|O(\delta^{1/4})|\le C\delta^{1/4}$ for the same constant.)
\end{lemma}
\begin{proof}
Throughout this proof, we will write $t_{k} =k\delta$ and denote by $C$ a generic constant that depends on the bounds
on $g,s,v$, and can change from line to line. Note that, by construction, $\hq_j=1$ for all $j$, and therefore 
 $(U^{\delta}_{j})_{j\ge 0}\sim_{iid}\normal(0,1)$. Hence  we can construct the discrete and 
continuous processes on the same space by letting $\sqrt{\delta}\Ude_{j} = B_{t_{j+1}}-B_{t_j}$. 

We then decompose the 
difference between the two processes as
\begin{align*}
X_{k\delta}-X_k^{\delta} = &\sum_{j=0}^{k-1}\int_{t_j}^{t_j+\delta}\big[v(X_t,t)-v(X_{j}^{\delta},t_{j+1})\big] \, \de t\\
&+ \sum_{j=0}^{k-1}\int_{t_j}^{t_j+\delta}\big[s(X_t,t)-s(X_{j}^{\delta},t_{j+1})\big] \, \de B_t\, .
\end{align*}
By taking the second moment, and using the fact that $X_t$ is measurable on $(B_s)_{s\le t}$ and $\Xde_j$ is measurable on $(B_s)_{s\le t_j}$,
we get
\begin{align*}
\E\big\{\big[X_{k\delta}-X_k^{\delta}\big]^2\big\}  \le &2k\sum_{j=0}^{k-1}\delta \int_{t_j}^{t_j+\delta}\E\big\{\big[v(X_t,t)-v(X_{j}^{\delta},t_{j+1}) \big]^2\big\} \, \de t\\
&+ 2\sum_{j=0}^{k-1}\int_{t_j}^{t_j+\delta}\E\big\{\big[s(X_t,t)-s(X_{j}^{\delta},t_{j+1})\big]^2\big\}\, \de t\, .
\end{align*}
Next notice that by the boundedness of $s, v$, we have $\E\{|X_t-X_s|^2\}\le C|t-s|$. 
Let $\Delta_k\equiv \E(|X_{t_k}-\Xde_{k}|^2)$.
Assuming without loss of generality $\delta<1$, 
\begin{align}
\E\big\{\big[v(X_t,t)-v(X_{j}^{\delta},t_{j+1}) \big]^2\big\}\le C\E(|X_t-X_{t_j}|^2) +C\E(|X_{t_j}-\Xde_{j}|^2)+C|t-t_{j+1}|^2 \le C\Delta_k+C\delta\, .
\end{align}
The same bound holds for $\E\{[s(X_t,t)-s(X_{j}^{\delta},t_{j+1}) ]^2\}$. Substituting above, we get
\begin{align}
\Delta_k \le C(\oq+1)\delta\sum_{j=0}^{k-1} (\Delta_j+\delta) .
\end{align}
This implies bound $\E(|X_{t_k}-\Xde_{k}|^2)\le C\delta$ as stated in \eqref{eq:Discr-1-a}.

In order to prove Eq.~\eqref{eq:Discr-1-b}, note that
\begin{align}
\E\{g(X_{k-1}^{\delta},t_{k})^2\}\le \E\{[g(X_{t_{k-1}},t_{k})+C|\Xde_{k-1}-X_{t_{k-1}}|]^2\}\le 1+C\sqrt{\delta}\, . 
\end{align}
Hence
\begin{align}
\E\big\{[\hg_{k}(X_{k-1}^{\delta})-g(X_{k-1}^{\delta},t_{k})]^2\big\} \le C\sqrt{\delta}\, .
\end{align}
Let $K=\lfloor \oq/\delta\rfloor$, and write
\begin{align}
Z_{K\delta}-Z^{\delta} = &\sum_{j=0}^{K-1}\int_{t_j}^{t_j+\delta}\big[g(X_t,t)-\hg_{j+1}(X_{j}^{\delta})\big] \, \de B_t\, .
\end{align}
Therefore
\begin{align*}
\E\big(|Z_{K\delta}-Z^{\delta}|^2\big)= &\sum_{j=0}^{K-1}\int_{t_j}^{t_j+\delta}\E\big\{[g(X_t,t)-\hg_{j+1}(X_{j}^{\delta})]^2\big\} \, \de t\\
\le & 2\sum_{j=0}^{K-1}\int_{t_j}^{t_j+\delta}\E\big\{[g(X_t,t)-g(X_{j}^{\delta},t_{j+1})]^2\big\} \, \de t\\
&+2 \sum_{j=0}^{K-1}\int_{t_j}^{t_j+\delta}\E\big\{[\hg_{j+1}(X_{j}^{\delta})-g(X_{j}^{\delta},t_{j+1})]^2\big\} \, \de t\\
\le & C\delta \sum_{j=0}^{K-1}\int_{t_j}^{t_j+\delta}(\Delta_j+\delta) \, \de t + C(\oq+1)\sqrt{\delta}\\
\le & C(\oq+1)\sqrt{\delta}\, .
\end{align*}
The bound of Eq.~\eqref{eq:Discr-1-b} follows since 
\begin{align}
\E\big(|Z_{q_*}-Z_{K\delta}|^2) = \int_{K\delta}^{q_*}\E\big\{g(X_t,t)^2\}\de t \le \delta\, .
\end{align}
Finally, Eq.~\eqref{eq:Discr-2} follows by the same estimates. 
\end{proof}

We now collect the main findings of this section in a theorem. This characterizes the values of the objective function achievable by the above algorithm.
\begin{theorem}\label{thm:AnalysisAMP}
Let $g,s,v:\reals\times\reals_{\ge 0}\to \reals$ be Lipschitz continuous, with $v$ and $s$ bounded, and define the process 
$(X_t,Z_t)$ using the SDE \eqref{eq:SDE} with initial condition $X_0=Z_0=0$.
Assume $\E\{g(X_t,t)^2\} = 1$ for all $t\ge 0$.  Further assume $\partial_x g(x,t) \partial_x s(x,t) \partial_x v(x,t)$ to exist and be Lipschitz continuous.

Define the incremental AMP iteration $(\bu^k)_{k\ge 0}$, and let $\bz$ be given by Eq.~\eqref{eq:z-def}. Finally, let $\psi:\reals\to\reals$ be a pseudo-Lipschitz
function. Then, for any $\eps>0$ there exist $\delta_*(\eps)>0$, and for any $\delta\ge \delta_*(\eps)$ there exist $M_*(\eps,\delta)<\infty$
such that, if $\delta\le \delta_*(\eps)$ and $M\ge M_*(\eps,\delta)$, we have
\begin{align}
\left|\plim_{n\to\infty}\frac{1}{2n}\<\bz,\bA\bz\>-\int_0^{\oq} \E\{g(X_t,t)\} \, \de t \right|& \le \eps\, ,\\
\left|\plim_{n\to\infty}\frac{1}{n}\sum_{i=1}^n\psi(z_i)- \E\{\psi(Z_{\oq})\} \right|& \le \eps\, .
\end{align}
(Further the above limits in probability are non-random quantities.)
\end{theorem}
\begin{proof}
This follows immediately from Lemma \ref{lemma:Zvalue} and Lemma \ref{lemma:SDE}.
\end{proof}

\section{Proof of the main theorem}
\label{sec:Proof}

\subsection{Choosing the nonlinearities}

In view of Theorem \ref{thm:AnalysisAMP}, we need to choose the coefficients $g,s,v$ in the SDE \eqref{eq:SDE} as to satisfy
two conflicting requirements: $(i)$ maximize $\int_0^{\oq} \E\{g(X_t,t)\} \, \de t$ (the energy value achieved by our algorithm);
$(ii)$ keep $\prob(Z_{q_*}\in [-1,1]) =1$ (we want a solution in the hypercube).

Throughout this section we set $\beta>\beta_0$ as per Assumption \ref{ass:FRSB}. We also set $q_*=q_*(\beta)$ and
$\mu=\mu_{\beta}$ the unique minimizer of the Parisi functional. We also fix $\Phi$ to be the solution of the PDE \eqref{eq:ParisiPDE}
with $\mu=\mu_*$.

There is a natural SDE associated with the Parisi's variational principle, that was first introduced in physics \cite{SpinGlass},
and recently  studied in the probability theory literature \cite{auffinger2015parisi,jagannath2016dynamic}:
\begin{equation}\label{eq:ParisiSDE}
\de X_t = \beta^2\mu(t)\partial_x\Phi(t,X_t)\, \de t + \beta\, \de B_t\,  .
\end{equation}
Unless otherwise stated, it is understood that we set the initial condition to $X_0=0$. Motivated by this, we set the coefficients $g,s,v$
as follows
\begin{align}
v(x,t) = \beta^2\mu(t)\partial_x\Phi(t,x)\, ,\;\;\;\; s(x,t) = \beta\, ,\;\;\;\; g(x,t) =\beta\partial_{xx}\Phi(t,x)\, . \label{eq:Choice}
\end{align}
We collect below a few useful regularity properties of $\Phi$, which have been proved in the literature.
\begin{lemma}\label{lemma:Reg}
\begin{itemize}
\item[$(i)$] $\partial_x^j\Phi(t,x)$ exists and is continuous for all $j\ge 1$.
\item[$(ii)$] For all  $(t,x)\in [0,1]\times \reals$, 
\begin{align}
\big|\partial_x \Phi(t,x)\big|\le 1\, ,\;\;\; 0<\partial_x^2 \Phi(t,x)\le 1\, ,\;\;\; \big|\partial_x^3 \Phi(t,x)\big|\le 4\,\, .
\end{align}
\item[$(iii)$] $\partial_t\partial_x^j\Phi(t,x) \in L^{\infty}([0,1]\times \reals)$ for all $j\le 0$.
\item[$(iv)$] $\partial_x\Phi(t,x)$, $\partial_{x}^2\Phi(t,x)$ are Lipschitz continuous on $[0,1]\times \reals$.
\end{itemize}
\end{lemma}
\begin{proof}
Points $(i)$ and $(iii)$ are Theorem 4 in \cite{jagannath2016dynamic}. Point $(ii)$ is Proposition 2.$(ii)$ in \cite{auffinger2015parisi}.
Finally, point $(iv)$ follows immediately from points $(iii)$, $(iv)$.
\end{proof}

This Lemma implies that the choice \eqref{eq:Choice} satisfies the regularity assumptions in Theorem \ref{thm:AnalysisAMP}.
We next have to check the normalization condition, and compute the resulting distribution.
\begin{lemma}
We have 
\begin{align}
Z_t = \partial_x\Phi(t,X_t)\, . \label{eq:Zsol}
\end{align}
In particular $\prob(Z_t\in [-1,1]) = 1$ for all $t$.
\end{lemma}
\begin{proof}
By Lemma 2 in \cite{auffinger2015parisi}, we have, for any $t_1<t_2$
\begin{align}
\partial_x\Phi(t_2,X_{t_2}) -\partial_x\Phi(t_1,X_{t_1}) = \int_{t_1}^{t_2} \beta\partial_{xx}\Phi(t,X_t)\, \de B_t\, ,\label{eq:DPhi-Int}
\end{align}
which is exactly Eq.~\eqref{eq:Zsol}. Lemma \ref{lemma:Reg}.$(ii)$ implies $|Z_t|\le 1$ almost surely.
\end{proof}

\begin{lemma}
For all $0\le t\le q_*$, we have
\begin{align}
\E\big\{\big(\partial_x\Phi(t,X_t)\big)^2\big\} &= t\, ,\label{eq:Dx2}\\
\E\big\{\big(\beta\partial_{xx}\Phi(t,X_t)\big)^2\big\}&= 1\, . \label{eq:Dxx2}
\end{align}
\end{lemma}
\begin{proof}
Equation \eqref{eq:Dx2} is Proposition 1 in \cite{chen2017variational}. For Eq.~\eqref{eq:Dxx2} note that by 
Eq.~(39) in the same paper, we have, for any $t_1<t_2\le q_*$
\begin{align}
\E\{(\partial_{x}\Phi(t_2,X_{t_2}))^2\}-\E\{(\partial_{x}\Phi(t_1,X_{t_1}))^2\} = \int_{t_1}^{t_2}\E\big\{\big(\beta\partial_{xx}\Phi(t,X_t)\big)^2\big\}\, \de t\, ,
\end{align}
and therefore the claim follows from Eq.~\ref{eq:Dx2}.
\end{proof}

\begin{lemma}\label{lemma:Dxx-Mu}
For any $0\le t\le q_*$, we have 
\begin{align}
\E\{\partial_{xx}\Phi(t,X_t)\} = \int_t^1\mu(s) \, \de s\, .
\end{align}
\end{lemma}
\begin{proof}
Consider $t\in[0,q_*]$ a continuity point of $\mu$. Then the proof of Lemma 16 in \cite{jagannath2016dynamic} yields
\begin{align}
\partial_{xx}\Phi(t,X_t) = 1-\mu(t)\big(\partial_x \Phi(t,X_t)\big)^2 -\E\left\{\int_t^1 \big(\partial_x \Phi(s,X_s)\big)^2\mu(\de s)\right\}\, ,
\end{align}
Taking expectation and using Fubini's alongside Eq.~(\ref{eq:Dx2}), we get
\begin{align}
\E\{\partial_{xx}\Phi(t,X_t)\} = 1-\mu(t)t-\int_{t}^1 s\,\mu(\de s) = \int_t^1\mu(s)\, \de s
\end{align}
The claim follows also for $t$ not a continuity point because the right hand side is obviously continuous in $t$.
The left hand side is continuous because $\partial_{xx}\Phi$ is Lipschitz (cf. Lemma \ref{lemma:Reg}) and $\E\{|X_t-X_s|^2\}\le C|t-s|$
because the coefficients of the SDE \label{ref:SDE} are bounded Lipschitz.
\end{proof}

We summarize the results of this section in the following theorem. Here and below, for $\bx\in\reals^n$, $S\subseteq\reals^n$,
we let $d(\bx,S) \equiv \inf\{|\bx-\by|\,:\; \by\in S\}$. 
\begin{theorem}\label{thm:AnalysisAMP-Parisi}
Under Assumption \ref{ass:FRSB} let $g,s,v:\reals\times\reals_{\ge 0}\to \reals$ be defined as per Eq.~\eqref{eq:Choice}, 
and set $\oq=q_*(\beta)$ for $\beta>\beta_0$.
Further let
\begin{align}
\cE(\beta) \equiv \frac{\beta}{2}[1-(1-q_*(\beta))^2] - \frac{\beta}{2}\int_0^1 s^2\, \mu_{\beta}(\de s)\, .
\end{align}
Define the incremental AMP iteration $(\bu^k)_{k\ge 0}$ via Eqs.~\eqref{eq:GeneralAMP}, \eqref{eq:IAMP-Def-1}, \eqref{eq:IAMP-Def-2},
with $\hg_k$ given by Eq.~\eqref{eq:hg-def}, and let $\bz$ be given by Eq.~\eqref{eq:z-def}. 
Then, for any $\eps>0$ there exist $\delta_*(\eps)>0$, and for any $\delta\ge \delta_*(\eps)$ there exist $M_*(\eps,\delta)<\infty$
such that, if $\delta\le \delta_*(\eps)$ and $M\ge M_*(\eps,\delta)$, we have
\begin{align}
\left|\plim_{n\to\infty}\frac{1}{2n}\<\bz,\bA\bz\>-\cE(\beta) \right|& \le \eps\, ,\\
\plim_{n\to\infty}\frac{1}{n} d(\bz,[-1,1]^n)^2& \le \eps\, .
\end{align}
(Further the above limits in probability are non-random quantities.)
\end{theorem}
\begin{proof}
First notice that $d(\bz,[-1,1]^n)^2 = \sum_{i=1}^n \psi(z_i)$ with $\psi(z_i) = d(z_i,[-1,1])^2$ a pseudo-Lipschitz function.
Further, integration by parts yields
\begin{align}
\cE(\beta) = \beta\int_0^{q_*}\int_t^1 \mu(s)\, \de s\, \de t\, .
\end{align}
Hence the claims of this theorem follow immediately from Theorem \ref{thm:AnalysisAMP} upon checking those assumptions
using the lemmas given in this section.
\end{proof}

\subsection{Sequential rounding and putting everything together}

Theorem \ref{thm:AnalysisAMP-Parisi} constructs a vector $\bz\in\reals^n$. It is not difficult to round this to a 
vector with entries in $\{+1,-1\}$, as detailed in the next lemma.
\begin{lemma}\label{lemma:Rounding}
There exist an algorithm with complexity $O(n^2)$, and an absolute constant $C>0$ such that the following happens with probability at least $1-e^{-n}$.
Given $\bA\sim\GOE(n)$ and a vector $\bx\in\reals^n$ such that $d(\bx,[-1,1]^n)^2\le n\, \eps_0$. Then there algorithm returns a vector
$\bsigma_*\in\{+1,-1\}^n$ such that
\begin{align}
\frac{1}{2n}\<\bsigma_*,\bA\bsigma_*\>\ge \frac{1}{2n}\<\bx,\bA\bx\>-20\Big(\sqrt{\eps_0}+\frac{1}{\sqrt{n}}\Big)\, .
\end{align}
\end{lemma}
\begin{proof}
Recall the definition of Hamiltonian $H_n(\bx) \equiv \<\bx,\bA\bx\>/2$ (which we view as a function on $\reals^n$). We also define 
$\tH_n(\bx) = H_n(\bx) -\sum_{i=1}^n A_{ii}x_i^2/2 = \sum_{i<j\le n} A_{ij}x_ix_j$.

We construct $\bsigma_*$ in two steps. First we let $\tbz$ to be the projection of $\bz$ onto the hypercube $[-1,+1]^n$ (i.e. $\tbz\in [-1,+1]^n$
is such that $|\tbz-\bz|^2 = d(\tbz,[-1,+1]^n)^2\le n\,\eps_0$). Note that this can be constructed in $O(n)$ time (simply by projecting each coordinate
$\tz_i$ onto $[-1,+1]$). 

Second, note that the function $\tH_n(\bx)$ is linear in each coordinate of $\bx$. Namely, for each $\ell$ 
$\tH_n(\bx) = x_{\ell} h_{1,\ell}(\bx_{\sim \ell};\bA)+ h_{0,\ell}(\bx_{\sim\ell};\bA)$, where $\bx_{\sim\ell} = (x_i)_{i\in [n]\setminus\ell}$ and $h_{1,\ell}(\bx_{\sim \ell};\bA) = 
\sum_{j\neq \ell}A_{\ell j}x_j$. We then construct a sequence
$\tbz(0),\dots \tbz(n)$ as follows. Set $\tbz(0) = \tbz$ and, for each $1\le \ell \le n$:
\begin{align}
\tbz(\ell)_{i} = \begin{cases}
\tbx(\ell-1)_i &\;\;\mbox{if $i\neq \ell$,}\\
\sign\big(h_{1,\ell}(\tbz(\ell-1)_{\sim \ell};\bA)\big) &\;\;\mbox{if $i =\ell$.}
\end{cases}
\end{align}
Finally we set $\bsigma_* = \tbz(n)$. This procedure takes $O(n^2)$ operations.

The lemma then follows straightforwardly from the following three claims:
\begin{itemize}
\item[$(i)$] $\tH_{n}(\bsigma_*)\ge \tH_n(\tbz)$.
\item[$(ii)$] $|\tH_{n}(\bsigma_*)-H_n(\bsigma_*)|\le 20\sqrt{n}$, $|\tH_{n}(\bsigma_*)-H_n(\bsigma_*)|\le 20\sqrt{n}$ with probability at least $1-e^{-2n}$.
\item[$(iii)$] $|H_n(\bz)-H_n(\tbz)|\le 20n\sqrt{\eps_0}$ with probability at least $1-e^{-2n}$.
\end{itemize}
Claim $(i)$ is immediate since $\tH_n(\tbz(\ell+1))\ge \tH_n(\tbz(\ell+1))$ for each $\ell$. 

Claim $(ii)$ holds since, for any $\bx\in [-1,+1]^n$, 
\begin{align}
|\tH_{n}(\bx)-H_n(\bx)|\le \frac{1}{2}\sum_{i=1}^n |A_{ii}| \equiv \tau(\bA)\, .
\end{align}
Now we have $\E\tau(\bA)=\sqrt{n/\pi}$, and $\tau$ is a Lipschitz function of the Gaussian vector $(A_{ii})_{i\le n}$. hence the desired bounds follow by Gaussian concentration.

For claim $(iii)$, let $\bv = \bz-\tbz$ and note that (denoting by $\lambda_{\max}(\bA)$ the maximum eigenvalue of $\bA$)
\begin{align}
\big|H_n(\bz)-H_n(\tbz)\big| &\le \frac{1}{2}|\<\bv,\bA\bv\>|+ |\<\bv,\bA\tbz\>|\\
&\le \frac{1}{2}\lambda_{\max}(\bA)|\bv|^2 + \lambda_{\max}(\bA)|\bv|\, |\tbz|\\
& \le n\lambda_{\max}(\bA)\Big[\frac{1}{2}\eps_0+\sqrt{\eps_0}\Big] \le 2n\lambda_{\max}(\bA) \sqrt{\eps_0}\, .
\end{align}
The desired probability bound follows by concentration of the largest eigenvalue of $\GOE$ matrices \cite{Guionnet}.
\end{proof}

We finally need to show that the quantity $\cE(\beta)$ of Theorem \ref{thm:AnalysisAMP-Parisi} converges to the asymptotic 
optimum value, for large $\beta$. This is achieved in the two lemmas below.
\begin{lemma}\label{lemma:E0}
Let $\cE_0(\beta) \equiv (\beta/2)(1-\int_0^1 t^2\, \mu_{\beta}(\de t))$. Then, almost surely,
\begin{align}
\cE_0(\beta)\le \lim_{n\to\infty}\frac{1}{2n}\max_{\bsigma\in\{+1,-1\}^n}\<\bsigma,\bA\bsigma\>\le \cE_0(\beta)+\frac{\log 2}{\beta}\, .
\end{align}
\end{lemma}
\begin{proof}
By Gaussian concentration, it is sufficient to consider the expectation  $E_n = 
\E\max_{\bsigma\in\{+1,-1\}^n}h_n(\bsigma)/n$ (recall that $H_n(\bsigma) = \<\bsigma,\bA\bsigma\>/2$.
Recall the definition of partition function $Z_n(\beta) = \sum_{\bsigma\in\{+1,-1\}^n}\exp(\beta H_n(\bsigma))$, and
define the associated Gibbs measure $\nu_{\beta}(\bsigma) = \exp(\beta H_n(\bsigma))/Z_n(\beta)$ and free energy density 
$F_n(T)\equiv (T/n)\E\log Z_n(\beta=1/T)$. A standard thermodynamic identity \cite{MezardMontanari} yields 
$F_n(T) = \E\nu_{1/T}(H_n(\bsigma))+ TS(\nu_{1/T})$, where $S(q)$ is the Shannon entropy of the probability distribution $q$. 
Further $F'_n(T) =S(\nu_{1/T})\ge $ and $F_n(T)\to E_n$ as $T\to 0$.
Hence
\begin{align}
\E\nu_{\beta}(H_n(\bsigma))\le E_n \le F_n(1/\beta) \le \E\nu_{\beta}(H_n(\bsigma))+\frac{\log 2}{\beta}\, .
\end{align}
On the other hand, $\partial_{\beta}(\beta F_n(\beta)) =  \E\nu_{\beta}(H_n(\bsigma))$. Since $\beta F_n(\beta)\to \Par_{\beta}(\mu_{\beta})$ by Theorem
\ref{thm:Talagrand}, $F_n(\beta), \Par_{\beta}(\mu_{\beta})$ are convex  with $\Par_{\beta}(\mu_{\beta})$ differentiable \cite{talagrand2006parisi-b}, it follows that
\begin{align}
\lim_{n\to\infty}\E\nu_{\beta}(H_n(\bsigma)) = \frac{\de\phantom{\beta}}{\de\beta} \Par_{\beta}(\mu_{\beta}) = \cE_0(\beta)\, .
\end{align}
(The last equality is proved in \cite{talagrand2006parisi-b}, with a difference in normalization of $\beta$.)
\end{proof}

\begin{lemma}\label{lemma:qstar}
For any $\beta>\beta_0$,
\begin{align}
\lim_{\beta\to\infty}\beta^2(1-q_*(\beta))^2\le 1\, .
\end{align}
\end{lemma}
\begin{proof}
The PDE \eqref{eq:ParisiPDE} can be solved for $t\in (q_*,1]$ using the Cole-Hopf transformation $\Phi = \log u$. This yields
$\Phi(q_*,x) = ((1-q_*)/2)+\log 2\cosh x$, whence $\partial_x\Phi(q_*,x) = \tanh(x)$ and $\partial_{xx}\Phi(q_*,x) = 1-\tanh(x)^2$.
Substituting in Eqs.~\eqref{eq:Dx2}, \eqref{eq:Dx2}, we get 
\begin{align}
\E\big\{\tanh(X_{q_*})^2\big\} &= q_*\, ,\\
\beta^2\E\big\{\big(1-\tanh(X_{q_*})^2\big)^2\big\}&= 1\, . 
\end{align}
Hence
\begin{align}
\beta^2(1-q_*)^2 = \beta^2\E\big\{1-\tanh(X_{q_*})^2\big\}^2\le
\beta^2\E\big\{\big(1-\tanh(X_{q_*})^2\big)^2\big\} =1\, .
\end{align}
\end{proof}

The proof our main result, Theorem \ref{thm:Main}, follows quite easily from the findings 
of this section.
\begin{proof}[Proof of Theorem \ref{thm:Main}]
Let $E_*\equiv \lim_{n\to\infty}\max_{\bsigma\in\{+1,-1\}^n}H_n(\bsigma)/n$. This limit exists by Corollary \ref{coro:T0},
and we further have $E_*\ge 1/2$ (this can be proved by the same thermodynamic argument as in the proof of Lemma \ref{lemma:E0},
noting that $(1/n)\log_n Z_n(\beta)\to \log 2+(\beta^2/4)$ for $\beta\le 1$ \cite{panchenko2013sherrington}).
It is therefore sufficient to output $\bsigma_*$ such that, with high probability, 
$H_n(\bsigma_*)/n\ge E_*-(\eps/3)$.

Let $\beta = 10/\eps$. By Lemma \ref{lemma:E0} and Lemma \ref{lemma:qstar}, we have $\cE(\beta)\ge E_*-(\eps/5)$.
Applying the algorithm of Theorem \ref{thm:AnalysisAMP-Parisi} thus we obtain, with high probability, a vector $\bx\in\reals^n$
such that $H_n(\bz)\ge E_*-\eps/4$ and $d(\bx,[-1,1]^n)^2\le \eps^2/10^6$. The proof is completed by using the rounding procedure of 
Lemma \ref{lemma:Rounding}.
\end{proof}

\section{Relation with the TAP equations}
\label{sec:TAP}

In this section we prove that the algorithm described in Section \ref{sec:Algo}, when used in conjunction with the specific 
choice of functions $g_k$, $s$, $v$ in Section \ref{sec:Proof} actually constructs an approximate solution of the TAP equations 
(under Assumption \ref{ass:FRSB}). As in the previous section, we set $\oq=q_*$, $v(x,t) = \beta^2\mu(t)\partial_x\Phi(t,x)$, 
$s(x,t) = \beta$, $g(x,t) =\beta\partial_{xx}\Phi(t,x)$, and 
\begin{align}
\hg_k(x) \equiv \frac{g(x,k\delta)}{\E\{g(\Xde_{k-1},k\delta)^2\}^{1/2}}\, .\label{eq:hg-def-bis}
\end{align}
Using these settings, we recall that $\bx^k$ and $\bz$ are given by
\begin{align}
\bx^k& = \bx^{k-1} + v(\bx^{k-1},k\delta)\, \delta + \beta\sqrt{\delta} [\bu_k]_M\sqrt{\delta}\, ,\label{eq:Xiter1}\\
\bz &= \sqrt{\delta}\sum_{k=1}^{\lfloor \oq/\delta\rfloor} g_k(\bu_0,\dots,\bu_{k-1})\odot\bu^{k}\,.\label{eq:Xiter2}
\end{align}
Finally, we will repeatedly use the fact that the PDE \eqref{eq:ParisiPDE}  can be solved on $(q_*,1]$ using the Cole-Hopf transformation, which yields 
$\Phi(q_*,x) = \log 2\cosh(x)+(1-q_*)/2$.

\begin{lemma}\label{lemma:TAP1}
Setting $k_*= \lfloor q_*/\delta\rfloor$, we have
\begin{align}
\lim_{\delta\to 0}\lim_{M\to\infty}\plim_{n\to\infty}\frac{1}{n}\left|\bz-\tanh(\bx^{k_*})\right|^2 = 0\, .
\end{align}
\end{lemma}
\begin{proof}
By Lemma~\ref{lemma:SE}, we have
\begin{align}
\lim_{M\to\infty}\plim_{n\to\infty}\frac{1}{n}\left|\bz-\tanh(\bx^{k_*})\right|^2 = \E\left\{\big[Z^{\delta}-\partial_x\Phi(q_*,X^{\delta}_{k_*})\big]^2\right\}\, .
\end{align}
On the other hand, using Lemma \ref{lemma:SDE}, we obtain
\begin{align}
\lim_{\delta\to 0} \E\left\{\big[Z^{\delta}-\partial_x\Phi(q_*,X^{\delta}_{k_*})\big]^2\right\}
&= \E\left\{\big[Z_{q_*}-\partial_x\Phi(q_*,X_{q_*})\big]^2\right\}\\
&= \E\left\{\left[\int_0^{q_*} \!\beta \partial_{xx}\Phi(t,X_{t})\,  \de B_t-\partial_x\Phi(q_*,X_{q_*})\right]^2\right\} = 0\, .
\end{align}
where the last identity follows from Eq.~\eqref{eq:DPhi-Int}.
\end{proof}

\begin{lemma}\label{lemma:TAP2}
Setting $k_*= \lfloor q_*/\delta\rfloor$, we have
\begin{align}
\lim_{\delta\to 0}\lim_{M\to\infty}\plim_{n\to\infty}\frac{1}{n}\left|\beta\bA\bz- \bx^{k_*}-\beta^2(1-q_*)\tanh(\bx^{k_*})\right|^2=0\, .
\end{align}
\end{lemma}
\begin{proof}
Throughout the proof, we will write $\bff_k \equiv f_k(\bu_0,\dots,\bu_k)$. By the basic iteration \eqref{eq:GeneralAMP}, we have
\begin{align}
\bA\bz = \sqrt{\delta}\sum_{k=1}^{k_*}\bA\bff_{k} = \sqrt{\delta}\sum_{k=1}^{k_*}\bu^{k+1} + \sqrt{\delta}\sum_{k=1}^{k_*}\sum_{\ell=1}^{k}\sb_{k\ell} \bff_{\ell-1}\,. 
\end{align}
Using Eqs.~\eqref{eq:Bkj} and \eqref{eq:Bkk}, together with the fact that $|\bff_{k}|^2/n$, $|\bu^k|^2/n$ are bounded by Lemma \ref{lemma:SE},
we get
\begin{align}
\lim_{M\to\infty}&\plim_{n\to\infty}\frac{1}{n}\left|\beta\bA\bz- \bx^{k_*}-\beta^2(1-q_*)\tanh(\bx^{k_*})\right|^2=\nonumber \\
&= 
\lim_{M\to\infty}\plim_{n\to\infty}\frac{1}{n}\left|\beta\sqrt{\delta}\sum_{k=1}^{k_*}\bu^{k+1} + \beta\sqrt{\delta}\sum_{k=1}^{k_*}
\E\{\hg_k(\Xde_{k-1}) \} \bff_{k-1}- \bx^{k_*}-\beta^2(1-q_*)\tanh(\bx^{k_*})\right|^2 \\
& = \E\left\{\left[\beta\sqrt{\delta}\sum_{k=1}^{k_*}\Ude_{k+1}+ \beta\sqrt{\delta}\sum_{k=1}^{k_*}
\E\{\hg_k(\Xde_{k-1}) \} \hg_{k-1}(\Xde_{k-2})\Ude_{k-1}- \Xde_{k_*}-\beta^2(1-q_*)\tanh(\Xde_{k_*})\right]^2\right\}\,. \label{eq:KeyTAP}
\end{align}
Next, using again Lemma \ref{lemma:SDE}, we have $\sqrt{\delta}\sum_{k=1}^{k_*}\Ude_{k+1}\toLtwo B_{q_*}$, 
$\Xde_{k_*}\toLtwo X_{q_*}$ and 
\begin{align}
 \sqrt{\delta}\sum_{k=1}^{k_*}
\E\{\hg_k(\Xde_{k-1}) \}\hg_{k-1}(\Xde_{k-2})\Ude_{k-1} &\toLtwo \int_0^{q_*} \E\{g(X_{t},t) \} \, g(X_t,t)\de B_t\\
&= \beta^2\int_0^{q_*} \E\{\partial_{xx}(t,X_{t}) \} \, \partial_{xx}\Phi(t,X_{t}) \de B_t\\
& = \beta^2\int_0^{q_*} \int_t^1\mu(s)\, \de s \, \partial_{xx}\Phi(t,X_{t}) \de B_t\, ,
\end{align}
where in the last step we used Lemma \ref{lemma:Dxx-Mu}. By Fubini's theorem
\begin{align}
 \beta^2\int_0^{q_*} \int_t^1\mu(s)\, \de s \, \partial_{xx}\Phi(t,X_{t}) \de B_t & = \beta^2\int_0^{q_*} \mu(s) \int_0^s \partial_{xx}\Phi(t,X_{t}) \de B_t \, \de s+
\beta^2\int_{q_*}^1 \mu(s) \int_0^{q_*} \partial_{xx}\Phi(t,X_{t}) \de B_t \, \de s\\
&= \beta \int_0^{q_*}\mu(s)\, \partial_x\Phi(X_s,s)\, \de s +\beta(1-q_*)  \partial_x\Phi(X_{q_*},q_*)\, ,
\end{align}
where in the last step we used once more Eq.~\eqref{eq:DPhi-Int}. Substituting these limits in Eq.~\eqref{eq:KeyTAP},
we get
\begin{align*}
\lim_{\delta\to 0}&\lim_{M\to\infty}\plim_{n\to\infty}\frac{1}{n}\left|\beta\bA\bz- \bx^{k_*}-\beta^2(1-q_*)\tanh(\bx^{k_*})\right|^2 =\\
&=
\E\left\{\left[\beta B_{q_*}+\beta^2 \int_0^{q_*}\mu(s)\, \partial_x\Phi(X_s,s)\, \de s +\beta^2(1-q_*)  \partial_x\Phi(X_{q_*},q_*)-X_{q_*}-
\beta^2(1-q_*)\tanh(X_{q_*})\right]^2\right\}\\
&= \E\left\{\left[\beta^2(1-q_*)  \partial_x\Phi(X_{q_*},q_*)-X_{q_*}-
\beta^2(1-q_*)\tanh(X_{q_*})\right]^2\right\} = 0\, .
\end{align*}
Where we used the fact that $X_t$ solves te SDE (\ref{eq:ParisiSDE}), and $\Phi(q_*,x) = \log 2\cosh(x)+(1-q_*)/2$.
\end{proof}

We can therefore state our result about constructing solutions to the TAP equations.
\begin{theorem}[Constructing solutions to the TAP equations]
Under Assumption \ref{ass:FRSB} let $g,s,v:\reals\times\reals_{\ge 0}\to \reals$ be defined as per Eq.~\eqref{eq:Choice}, 
and set $\oq=q_*(\beta)$ for $\beta>\beta_0$.
Define the incremental AMP iteration $(\bu^k)_{k\ge 0}$ via Eqs.~\eqref{eq:GeneralAMP}, \eqref{eq:IAMP-Def-1}, \eqref{eq:IAMP-Def-2},
with $\hg_k$ given by Eq.~\eqref{eq:hg-def}, and let $\bz$ be given by Eq.~\eqref{eq:z-def}. 
(The same iteration is given explicitly in Eqs.~\eqref{eq:Xiter1}, \eqref{eq:Xiter2}.) 

Set $k_*= \lfloor q_*/\delta\rfloor$.
Then, for any $\eps>0$ there exist $\delta_*(\eps)>0$, and for any $\delta\ge \delta_*(\eps)$ there exist $M_*(\eps,\delta)<\infty$
such that, if $\delta\le \delta_*(\eps)$ and $M\ge M_*(\eps,\delta)$, we have, with high probability
\begin{align}
\frac{1}{n}\left|\beta\bA\tanh(\bx^{k_*})- \bx^{k_*}-\beta^2(1-q_*)\tanh(\bx^{k_*})\right|\le \eps\, .
\end{align}
\end{theorem}
\begin{proof}
The theorem follows immediately from Lemma \ref{lemma:TAP1} and Lemma \ref{lemma:TAP2}, using the fact that,
with high probability, $\bA$ has operator norm bounded by $2+\eps$ \cite{Guionnet}.
\end{proof}

\section{Universality}
\label{sec:Universality}

In this section we use the universality results of \cite{bayati2015universality} to generalize Theorem \ref{thm:Main}
to other random matrix distributions. Namely, we will work under the following assumption:
\begin{assumption}\label{ass:Matrix}
The matrix $\bA=\bA(n)$ is symmetric with $A_{ii}=0$ and $(A_{ij})_{1\le i<j\le n}$ a collection of independent random variables,
satisfying $\E\{A_{ij}\}=0$, $\E\{A_{ij}^2\}=1/n$. Further, the entries are subgaussian, with common subgaussian parameter $C_*/n$.
(Namely, $\E\{\exp(\lambda A_{ij})\} \le \exp(C_*\lambda^2/2n)$ for all $i<j\le n$.)
\end{assumption} 

Using \cite[Theorem 4]{bayati2015universality}, and proceeding exactly as for Proposition \ref{propo:GeneralAMP},
we obtain the following.
\begin{proposition} \label{propo:GeneralAMP-Univ}
Consider the AMP iteration \eqref{eq:GeneralAMP}, with $\bA=\bA(n)$ satisfying Assumption \ref{ass:Matrix}.
Further, assume $f_k:\reals^{k+2}\to \reals$ to be a fixed polynomial (independent of $n$).
Then for any $k\in\naturals$, and any pseudo-Lipschitz function $\psi:\reals^{k+2}\to \reals$, we have
\begin{align}
\frac{1}{n}\sum_{i=1}^n\psi(u_i^0,\dots,u_i^k;y_i) \toP \E \psi(U_0,\dots,U_k;Y)\, .\label{eq:LimitGeneralAMP-Gen}
\end{align}
Here $(U_j)_{j\ge 1}$ is a centered Gaussian process independent of $(U_0,Y)$ with covariance $\hbQ = (\hQ_{kj})_{k,j\ge 1}$ determined recursively 
via
\begin{align}
\hQ_{k+1,j+1} = \E\big\{f_k(U_0,\dots,U_k;Y) f_j(U_0,\dots,U_j;Y)\big\}\, . \label{eq:SE-AMP-Gen}
\end{align}
\end{proposition}
Notice an important difference with respect to Proposition \eqref{eq:GeneralAMP}: instead of Lipschitz functions, we require the functions 
$f_k$ to be polynomials. However, this result is strong enough to allow us prove the following generalization of Theorem \ref{thm:Main}.
\begin{theorem}\label{thm:Universality}
Let $\bA=\bA(n)$, $n\ge 1$ be random matrices satisfying Assumption \ref{ass:Matrix}.
Under Assumption \ref{ass:FRSB}, for any $\eps>0$ there exists an algorithm that takes as input the matrix $\bA\in\reals^{n\times n}$, and
outputs $\bsigma_*=\bsigma_*(\bA)\in\{+1,-1\}^n$, such that the following hold:
$(i)$ The complexity (floating point operations) of the algorithm is at most $C(\eps)n^2$.
$(ii)$ We have $\<\bsigma_*,\bA\bsigma_*\>\ge (1-\eps) \max_{\bsigma\in\{+1,-1\}^n} \<\bsigma,\bA\bsigma\>$.
\end{theorem}
\begin{proof}
Let $\hg_k(x)$, $v(x,t)$, $s(x,t)$ be defined as in the proof of Theorem \ref{thm:Main} for $k\le 1/\delta$. 
For each $M\in \integers$,  and each $k\le 1/\delta$, we construct a polynomial  $\hp_{k,M}:\reals^{k-1}\to\reals$
which approximates the dynamics defined by $\hg_k(\, \cdot\,)$, $v(\, \cdot\,,k\delta)$, $s(\,\cdot\,,k\delta)$, in a sense that we will make precise below.

We define the IAMP iteration, analogously to \eqref{eq:IAMP-Def-1}, \eqref{eq:IAMP-Def-2}
\begin{align}
f_k(u_0,\dots,u_k) &= \hp_{k,M}\left(u_1,\dots,u_{k-1} \right)\, \cdot u_k\, ,\label{eq:IAMP-Def-1-Poly}
\end{align}
We then claim that we can construct these polynomial approximations $\hp_{k,M}$ so that,  
for any $k\le 1/\delta$, and any pseudo-Lipschitz function $\psi:\reals^{k+2}\to \reals$, we have
\begin{align}
\lim_{M\to\infty}\plim_{n\to\infty}\frac{1}{n}\sum_{i=1}^n\psi(u_i^0,\dots,u_i^k)  =\E \psi(\Ude_0,\dots,\Ude_k)\, ,\label{eq:ClaimUniversality}
\end{align}
where the independent random variables $(\Ude_\ell)_{\ell\ge 0}$ are defined as in Lemma \ref{lemma:SE}. Given this claim,
the rest of the proof of Theorem \ref{thm:Main} can be applied verbatimly to this --slightly different-- algorithm.

In order to prove the claim (\ref{eq:ClaimUniversality}), we proceed as in the proof of Lemma \ref{lemma:SE}. Namely, by applying Proposition
\ref{propo:GeneralAMP-Univ}, we get 
\begin{align}
\plim_{n\to\infty}\frac{1}{n}\sum_{i=1}^n\psi(u_i^0,\dots,u_i^k)  =\E \psi(U_0,\dots,U^{\delta,M}_k)\, ,
\end{align}
where $(U^{\delta,M}_{\ell})_{\ell\ge 0}$ is a centered Gaussian process. Using the same argument as in Lemma \ref{lemma:SE},
we obtain that the Gaussian random variables $(U^{\delta,M}_{\ell})_{\ell\ge 0}$ are independent. Further, letting $\hq^M_{\ell} \equiv \E\{(U^{\delta,M}_{\ell})^2\}$, 
Proposition  \ref{propo:GeneralAMP-Univ} yields the following recursion
\begin{align}
\hq^M_{k+1}&= \E\{\hp_{k,M}(U^{\delta,M}_{1},\dots, U^{\delta,M}_{k-1})^{2}\} \cdot \hq^M_{k}\, . 
\end{align}
The claim \eqref{eq:ClaimUniversality} follows by showing the we can choose polynomials $(\hp_{\ell,M})_{\ell\ge 0}$ so that 
$\lim_{M\to\infty}\hq^M_\ell = \hq_\ell$ for each $\ell\le 1/\delta$. This can be done by induction over $k$.  
As a preliminary, notice that there is $c_0=c_0(\delta)>0$ sufficiently small so that, for the sequence of random variables defined recursively via Eq.~\eqref{eq:SE-IAMP},
we have $2c_0\le \hq_k\le 1/(2 c_0)$ for all $k\le 1/\delta$  (the existence of such $c_0>0$ can also be shown by induction over $k$ using the fact that 
$\hg_k, v,s,$ are bounded Lipschitz).

The basis of the induction $\lim_{M\to\infty}\hq^M_0 = \hq_0$ is trivial. 
Then assume that the induction claim is true for all $\ell\le k$. Without loss of generality we can consider that, for any $M\ge 1$ 
we have $c_0\le \hq_1^M,\dots, \hq_k^M\le 1/c_0$. Indeed by the induction hypothesis this holds for all $M$ large enough, and we can always 
renumber the polynomials  $\hp_{\ell,M}(\,\cdots\,)$ so that it holds for all $M\ge 1$. Then notice that the random variable $\Xde_k$
of Eq.~\eqref{eq:SE-IAMP} can be written as $\Xde_k = h_k(U_0,U_1^{\delta},\dots,U_{k-1}^{\delta})$ for a certain function $h_k$ that is bounded by a polynomial. 
We then choose the polynomials $\hp_{k,M}(\,\cdot\,)$ so that
\begin{align}
\E\big\{\Big|h_k(U_0,U^{\delta,M}_{1},\dots, U^{\delta,M}_{k-1})-\hp_{k,M}(U_0,U^{\delta,M}_{1},\dots, U^{\delta,M}_{k-1})\big|^2\Big\}\le \frac{1}{M}\, .
\end{align}
Such polynomials can be constructed, for instance, by considering the
expansion of $h_k$ in the basis of multivariate Hermite polynomials (suitably rescaled as to form an orthonormal basis with in $L^2(\reals^{k-1},\mu_k)$,
where $\mu_k$ is the joint distribution of $U_0,U^{\delta,M}_{1},\dots, U^{\delta,M}_{k-1}$.) The variance bound $c_0\le \hq_1^M,\dots, \hq_k^M\le 1/c_0$
is used in controlling the error term.

The  induction claim then follows by 
\begin{align}
\lim_{M\to\infty}\E\{\hp_{k,M}(U^{\delta,M}_{1},\dots, U^{\delta,M}_{k-1})^{2}\}  &= \lim_{M\to\infty}\E\{h_k(U^{\delta,M}_{1},\dots, U^{\delta,M}_{k-1})^{2}\}  =
\E\{h_k(U^{\delta}_{1},\dots, U^{\delta}_{k-1})^{2}\} \, ,
\end{align}
where the last equality holds by dominated convergence.
\end{proof}

Corollary \ref{coro:MaxCut} follows by applying Theorem \ref{thm:Universality} with $\bA$ a suitably centered and normalized adjacency matrix.
\begin{proof}[Proof of Corollary \ref{coro:MaxCut}]
Given a graph $G\sim \cG(n,p)$, construct the matrix $\bA=\bA^{\sT}\in\reals^{n\times n}$, by setting $A_{ii}=0$ and, for $i\neq j$:
\begin{align}
A_{ij} = \begin{cases}
-\sqrt{\frac{1-p}{np}}& \;\; \mbox{if $(i,j)\in E$,}\\
\sqrt{\frac{p}{n(1-p)}}& \;\; \mbox{if $(i,j)\not\in E$,}
\end{cases}
\end{align}
It is easy to verify that this matrix satisfies Assumption \ref{ass:Matrix}. Further, we have
\begin{align}
\cut_G(\bsigma) = \frac{1}{2}|E_n|-\frac{p}{4}\<\bsigma,\bfone\>^2+\frac{1}{4}\sqrt{np(1-p)}\, \<\bsigma,\bA\bsigma\>\, .
\end{align}
Recall that we know from  \cite{dembo2017extremal} $\max_{\bsigma\in\{+1,-1\}^n}\cut_G(\bsigma)= |E_n|/2+(n^{3}p(1-p)/2)^{1/2}\Par_*+o(n^{3/2})$.
Let $\bsigma_1$ denote the output of the algorithm of Theorem \ref{thm:Universality}, on input $\bA$. 
Applying this theorem and Lemma  \ref{propo:GeneralAMP-Univ}, we get
\begin{align}
\plim\inf_{n\to\infty}\frac{1}{2n}\<\bsigma_1,\bA\bsigma_1\> &\ge (1-\eps)\Par_*\, ,\\
\plim_{n\to\infty}\frac{1}{n}\<\bsigma_1,\bfone\> & = 0\, .
\end{align}
We construct $\bsigma_*$ by balancing $\bsigma_1$. Namely, if $|\<\bsigma_1,\bfone\>| = \ell$, we obtain $\bsigma_*$ by flipping 
$\lfloor \ell/2\rfloor$ entries of $\bsigma_1$ so that $|\<\bsigma_*,\bfone\>| \le 1$. We then have, with high probability
\begin{align*}
\cut_G(\bsigma_*) &\ge  \frac{1}{2}|E_n| -\frac{p}{4}+\frac{1}{4}\sqrt{np(1-p)}\, \<\bsigma_*,\bA\bsigma_*\>-
\frac{1}{4}\sqrt{np(1-p)}\, \left|\<\bsigma_*,\bA\bsigma*\>-\<\bsigma_1,\bA\bsigma_1\>\right|\\
&\ge  \frac{1}{2}|E_n| +\frac{1}{4}(1-\eps)\sqrt{np(1-p)}\, \max_{\bsigma\in \{+1,-1\}^n}\<\bsigma,\bA\bsigma\>
-\sqrt{n}\|\bA\|_{\op}|\bsigma_1|\, |\bsigma_*-\bsigma_1|\, .
\end{align*}
(Here $\|\bA\|_{\op}$ denotes the operator norm of matrix $\bA$.)
Therefore, since $|\<\bsigma_1,\bfone\>|/n = \ell/n\toP 1$, and $\|\bA\|_{\op}\le 2.01$ with high probability \cite{Guionnet}, we get
\begin{align*}
\cut_G(\bsigma_*)-\frac{|E_n|}{2}& \ge (1-\eps) \max_{\bsigma\in \{+1,-1\}^n}\Big\{\cut_G(\bsigma)-\frac{|E_n|}{2}\Big\}- 
n\sqrt{\ell}\|\bA\|_{\op}\\
&\ge (1-\eps) \max_{\bsigma\in \{+1,-1\}^n}\Big\{\cut_G(\bsigma)-\frac{|E_n|}{2}\Big\}- o(n^{3/2})\, ,
\end{align*}
which completes the proof.
\end{proof}

\section*{Acknowledgements}

I am grateful to Eliran Subag for an inspiring presentation of his work \cite{subag2018following} delivered at  the workshop `Advances in Asymptotic Probability'
in Stanford, and for a stimulating conversation.
This work was partially supported by grants NSF DMS-1613091,  CCF-1714305,
IIS-1741162 and ONR N00014-18-1-2729.

\appendix

\section{Proof of Proposition \ref{propo:GeneralAMP}}
\label{app:GeneralAMP}

As mentioned in the main text, Proposition  \ref{propo:GeneralAMP} is a consequence of the general analysis of AMP algorithms
available in the literature. In particular it can be obtained from a reduction to the setting of \cite[Theorem 1]{javanmard2013state}.
Let us briefly recall the class of algorithms considered in \cite{javanmard2013state}, adapting the notations to the present ones.
(we limit ourselves to consider the `one-block' case in the language of \cite{javanmard2013state}).

Fixing $T\ge 1$ consider a sequence of Lipschitz functions 
\begin{align*}
F_t:&\reals^{T}\times \reals^2\to\reals^{T} \, ,\\
&(x_1,\dots,x_T,z_1,z_2)\mapsto F_t(x_0,x_1,\dots,x_T,z_1,z_2)\, .
\end{align*}
Given two matrices $\bx\in\reals^{n\times (T+1)}$, $\bz\in\reals^{n\times 2}$, we let $F_t(\bx;\bz)\in\reals^{n\times (T+1)}$
be the matrix whose $i$-th row is given by $F_t(\bx_i,\bz_i)$ (where $\bx_i$ is the $i$-th row of $\bx$ and $\bz_i$ is the $i$-th row of $\bz$).

Then \cite{javanmard2013state} analyzes the following AMP iteration, which produces a sequence of iterates $\bx^t\in \reals^{n\times(T+1)}$
\begin{align}
\bx^{t+1} = \bA F_t(\bx^t;\bz) - F_{t-1}(\bx^{t-1};\bz)\sB_{t}^{\sT}\, . \label{eq:AMP-Matrix}
\end{align}
Here $\sB_t\in\reals^{T\times T}$ is a matrix with entries defined by 
\begin{align}
(\sB_t)_{ij} = \frac{1}{n}\sum_{\ell=1}^n(D_{\bx}F_t(\bx^t_\ell;\bz_\ell))_{ij} = \frac{1}{n}\sum_{\ell=1}^n\frac{\partial F_t}{\partial x^t_{\ell,j}}(\bx^t_\ell;\bz_\ell)\, .
\end{align}
Under the assumption that $\bx^0,\bz$ are independent of $\bA$, and $\hp_{\bx^0,\bz}\equiv n^{-1}\sum_{i=1}^n\delta_{\bx^0_i,\bz_i}$ converges in $W_2$,
 \cite[Theorem 1]{javanmard2013state} determines the asymptotic empirical distribution of $\bx^t,\bz$. 

Proposition \ref{propo:GeneralAMP} can be recast as a special case of this setting. First notice that we can always choose an $n$-independent
$T$ such that the time horizon $k$ in  Eq.~\eqref{eq:LimitGeneralAMP} satisfies $k\le T$. We then consider the iteration \eqref{eq:AMP-Matrix}
with initialization $\bx^0=\bfzero$, data vectors  $\bz=(\bu^0,\by)$, and update functions given by 
\begin{align}
F_t(x_1,x_2,\dots, x_T,z_1,z_2)_{\ell} = f_{\ell-1}(z_1,x_1,\dots,x_{\ell-1};z_2)\, .
\end{align}
With this setting, the vector $(x^t_{i,\ell})_{i\le n}\in\reals^n$ coincides with $\bu^{\ell}$ as given in Eq.~\eqref{eq:GeneralAMP}, for all $t\ge \ell$.
The recursion of Eq.~\eqref{eq:SE-AMP} follows from the analogous recursion in \cite[Theorem 1]{javanmard2013state}.

\section{A simplified version of the algorithm}
\label{app:Simplified}

In this appendix we provide a simplified version of the algorithm of Theorem \ref{thm:Main}, for the reader's convenience.
In this presentation we simplify certain technical details that have been introduced in the main text to simplify the proof.
In the pseudo-code below $\odot$ denotes entrywise multiplication between vectors. Further, when a scalar function is applied to a vector,
it is understood to be applied componentwise. In particular, note that  $|\partial_{xx}\Phi(k\delta,\bx^{k})|$ is the $\ell_2$ norm of the vector whose $i$-th component is
$\partial_{xx}\Phi(k\delta,x_i^{k})$.

\begin{algorithm}[H]
\SetAlgoLined
\KwData{Matrix $\bA\sim \GOE(n)$, parameters $\delta$, $\beta>0$}
\KwResult{Near optimum $\bsigma_*\in \{+1-1\}^n$ of the SK Hamiltonian}
Compute minimizer $\mu_{\beta}$ of the Parisi functional $\Par_{\beta}(\mu)$ (cf. Eq.~\eqref{eq:ParisiFunctional})\;
Compute solution $\Phi$ PDE \eqref{eq:ParisiPDE}, with $\mu = \mu_{\beta}$\;
Compute $q_*(\beta) = \sup\{q:\; q\in\supp(\mu_{\beta})\}$ (Edwards-Anderson parameter)\;
Initialize $\bu^{-1}=\bfzero$, $\bu^{0}\sim\normal(0,\id_n)$, $\bg^{-1}=\bfone$, $\bg^{-2}=\bfzero$, $\sb_0=0$\;
\For{$k\leftarrow 0$ \KwTo $\lfloor q_*/\delta\rfloor$}{
$\bu^{k+1} = \bA (\bg^{k-1}\odot\bu^k)-\sb_k\bg^{k-2}\odot\bu^{k-1}$\;
$\bx^k = \bx^{k-1}+ \beta^2\mu(k\delta)\,\partial_x\Phi(k\delta,\bx^{k-1})\, \delta +\beta\sqrt{\delta} \bu^k$\;
$\bg^{k} = \sqrt{n}\partial_{xx}\Phi(k\delta,\bx^{k})/|\partial_{xx}\Phi(k\delta,\bx^{k})|$\;
$\sb_{k+1}=\sum_{i=1}^ng^{k}_i/n$\;
}
Compute $\bz = \sqrt{\delta}\sum_{k=1}^{\lfloor q_*/\delta\rfloor}\bg^{k-1}\odot\bu^k$\;
Round $\bz$ to $\bsigma_*\in\{-1,+1\}^n$\;
\KwRet $\bsigma_*$
\caption{IAMP algorithm to optimize SK Hamiltonian}
\end{algorithm}

Notice that this pseudo-code does not describe how to minimize the Parisi functional and to solve the
PDE \eqref{eq:ParisiPDE}. As discussed in the introduction, we believe this can be done efficiently because of
the strong convexity and continuity of $\mu\mapsto \Par_{\beta}(\mu)$. Indeed highly accurate numerical
solutions (albeit with no rigorous analysis) 
were developed already in \cite{crisanti2002analysis,oppermann2007double,schmidt2008method}.

Further, the pseudo-code does not specify the rounding procedure, which is given below.

\begin{algorithm}[H]
\SetAlgoLined
\KwData{Matrix $\bA\in\reals^{n\times n}$, vector $\bz\in\reals^n$}
\KwResult{Integer solution $\bsigma_*\in \{+1-1\}^n$}
\For{$i\leftarrow 1$ \KwTo $n$}{
Set $\tz_i \leftarrow \min(\max(z_i,-1),+1)$\;
}
\For{$i\leftarrow 1$ \KwTo $n$}{
Compute $h_i = \sum_{j\neq i}A_{ij}\tz_j$\;
Set $\tz_i \leftarrow \sign(h_i)$;
}
\KwRet $\bsigma_*=\tbz$.
\caption{Round}
\end{algorithm}

\bibliographystyle{amsalpha}
\newcommand{\etalchar}[1]{$^{#1}$}
\providecommand{\bysame}{\leavevmode\hbox to3em{\hrulefill}\thinspace}
\providecommand{\MR}{\relax\ifhmode\unskip\space\fi MR }
\providecommand{\MRhref}[2]{%
  \href{http://www.ams.org/mathscinet-getitem?mr=#1}{#2}
}
\providecommand{\href}[2]{#2}

\end{document}